\documentclass[11pt]{article}
\usepackage{amsfonts,amsmath,amsthm,amssymb}
\usepackage{hyperref}
\usepackage[enableskew]{youngtab}
\newtheorem{theorem}{Theorem}
\newtheorem*{thm:repeat}{Theorem \ref{thm:cosets-span}}
\newtheorem{conj}{Conjecture}
\newtheorem{lemma}{Lemma}
\newtheorem{claim}{Claim}
\newtheorem{Observation}{Observation}
\newtheorem*{fact}{Fact}
\newtheorem{definition}{Definition}
\newtheorem{cor}{Corollary}
\newtheorem{remark}{Remark}

\newcommand{\tr}{\textrm{Tr}}

\newcommand{\fpf}{\textrm{FPF}}
\newcommand{\sgn}{\textrm{sgn}}
\newcommand{\Span}{\textrm{Span}}
\newcommand{\Shape}{\textrm{Shape}}
\newcommand{\Split}{\textrm{Split}}
\newcommand{\domleq}{\unlhd}
\newcommand{\domgeq}{\unrhd}
\newcommand{\domless}{\lhd}

\newcommand{\Cay}{\textrm{Cay}}
\newcommand{\even}{\textrm{ev}}
\newcommand{\odd}{\textrm{odd}}

\newcommand{\yodd}{Y_{\textrm{\scriptsize{odd}}}}
\newcommand{\yeven}{Y_{\textrm{\scriptsize{even}}}}
\newcommand{\zodd}{Z_{\textrm{\scriptsize{odd}}}}
\newcommand{\zeven}{Z_{\textrm{\scriptsize{even}}}}

\newcommand{\A}[0]{{\cal A}}
\newcommand{\B}[0]{{\cal B}}
\newcommand{\Ell}[0]{{\cal L}}

\newcommand{\ra}[0]{\rightarrow}
\newcommand{\enote}[1]{{\bf (Ehud:} {\bf #1}{\bf ) }}
\newcommand{\hnote}[1]{{\bf (Haran:} {\bf #1}{\bf ) }}

\newcommand{\remove}[1]{}
\author{David Ellis\thanks{DPMMS, University of Cambridge.},
Ehud Friedgut\thanks{Hebrew
University, Jerusalem, and University of Toronto.
Research supported in part by the Israel Science Foundation, grant
no. 0397684, and NSERC grant 341527.}
and Haran Pilpel\thanks{Hebrew
University, Jerusalem. Research supported in part by the Giora Yoel Yashinsky
memorial grant. Current affiliation: Google, Inc.}}
\title{Intersecting Families of Permutations}
\begin{document}
\date{}
\renewcommand{\thefootnote}{\fnsymbol{footnote}}
\maketitle
\begin{abstract}
  A set of permutations $I \subset S_n$
  is said to be $k$-{\em intersecting} if any two permutations in \(I\)
  agree on at least \(k\) points. We show that for any $k \in \mathbb{N}$, if \(n\) is sufficiently large depending on \(k\), then the largest $k$-intersecting subsets of $S_n$ are cosets of stabilizers of \(k\) points, proving a conjecture of Deza and
  Frankl. We also prove a similar result concerning \(k\)-cross-intersecting subsets.
  Our proofs are based on eigenvalue techniques and the representation
  theory of the symmetric group.
\end{abstract}
\section{Introduction}
The classical Erd\H os-Ko-Rado theorem states that if $r < n/2$, an intersecting family of
 $r$-subsets of $\{1,2,\ldots,n\}$ has size at most ${n-1 \choose r-1}$;
 if equality holds, the family must consist of all $r$-subsets containing a fixed element.
 The so-called `second Erd\H os-Ko-Rado theorem' states that if $n$ is sufficiently large depending on
  $k$ and $r$, then any $k$-intersecting family of $r$-subsets of $\{1,2,\ldots,n\}$ has size at most
   ${n-k \choose r -k}$; if equality holds, the family must consist of all $r$-subsets containing $k$ fixed elements.
    We deal with analogues of these results for permutations.

 As usual, \([n]\) will denote the set \(\{1,2,\ldots,n\}\), and \(S_n\) will denote the symmetric group, the group of all permutations of \([n]\). Two permutations $\sigma,\tau \in S_n$ are said to {\em intersect} if they agree at some point, i.e. if there exists \(i \in [n]\) such that \(\sigma(i)=\tau(i)\). Similarly, they are said to $k$-{\em intersect}
if they agree on at least $k$ points, i.e., if there exist $i_1,i_2,\ldots,i_k \in
[n]$ such that $\sigma(i_t) = \tau(i_t)$ for $t= 1, 2,\ldots, k$. A
subset $I \subset S_{n}$ is said to be $k$-{\em intersecting} if any two permutations
in $I$ $k$-intersect. In this paper, we characterize the largest
$k$-intersecting subsets of \(S_n\) for \(n\) sufficiently large depending on \(k\), proving a conjecture of Deza and Frankl. We also prove a similar result concerning \(k\)-cross-intersecting subsets, proving a conjecture of Leader.

Our main tool in this paper is Fourier analysis on the symmetric
group, which entails representation theory. Although Fourier analysis has become a central tool in
combinatorics and computer science in the last two decades
(notably, since the landmark paper of \cite{kkl}), it has not often been
applied to combinatorial problems in a non-Abelian setting. In
retrospect, it seems that in this case it fits the task perfectly.

As a bonus, we point out a nice aspect of Boolean (0/1 valued) functions on $S_n$.
One of the recurring themes in the applications of discrete Fourier
analysis to combinatorics over the last decade has been showing that Boolean functions on \(\{0,1\}^{n}\) are `juntas' (i.e. depend essentially on few coordinates) precisely when their Fourier transform is concentrated mainly on small sets. Here, we study Boolean functions on $S_n$ whose Fourier transform is supported on the irreducible representations of low dimension,
and connect them to cosets of subgroups which are the pointwise stabilizers of small subsets of $\{1,2\ldots,n\}$.
Along the way, we also prove an interesting generalization of Birkhoff's theorem
on bistochastic matrices.

\subsection{History and related results}
\label{sec:history}
Let
\[T_{i \mapsto j} = \{ \sigma \in S_n,\ \sigma(i)=j\}.\]
Clearly, $T_{i \mapsto j}$ is 1-intersecting subset of \(S_n\), with size $(n-1)!$. Let
\[T_{i_1 \mapsto j_1, \ldots, i_k \mapsto j_k}
= \bigcap_{t=1}^k T_{i_t \mapsto j_t} = \{\sigma \in S_n:\ \sigma(i_t)=j_t\ (1 \leq t \leq k)\}.\]
If \(i_1,\ldots,i_k\) are distinct and \(j_1,\ldots,j_k\) are distinct, then \(T_{i_1 \mapsto j_1, \ldots, i_k \mapsto j_k}\) is a coset of the stabilizer of \(k\) points; we will refer to it as a $k$-{\em coset}.

Clearly, a $k$-coset is a $k$-intersecting family of size \((n-k)!\). As observed by Deza and Frankl \cite{deza-frankl}, it is easy to prove that a 1-intersecting subset of \(S_n\) is no larger than a 1-coset:
\begin{theorem} \label{one-intersecting}
  \cite{deza-frankl} For any \(n \in \mathbb{N}\), if \(I \subset S_n\) is $1$-intersecting, then $|I| \leq (n-1)!$.
\end{theorem}
\begin{proof}
  Let $H$ be the cyclic group generated by the $n$-cycle $(123\ldots n)$.
   No two permutations in $H$ intersect, and the same is true for any left coset of
   $H$. The $(n-1)!$ left cosets of $H$ partition $S_{n}$. If $I \subset S_n$ is 1-intersecting, then \(I\) contains at most one permutation from each left coset of \(H\), and therefore $|I| \leq (n-1)!.$
\end{proof}

Deza and Frankl conjectured in \cite{deza-frankl} that for any \(n \in \mathbb{N}\), the 1-cosets are the only 1-intersecting subsets of \(S_n\) with size \((n-1)!\). Perhaps surprisingly, this turned out to be substantially harder to prove; it was first proved by Cameron and Ku \cite{cameron-ku} and independently by Larose and Malvenuto \cite{larose-malvenuto}.

What about \(k\)-intersecting families? For \(n\) small depending on \(k\), the \(k\)-cosets need not be the largest \(k\)-intersecting subsets of \(S_n\). Indeed,
\[\{\sigma \in S_{n}:\ \sigma \textrm{ has at least }k+1\textrm{ fixed points in }\{1,2,\ldots,k+2\}\}\] 
is a \(k\)-intersecting family with size
\[(k+2)(n-k-1)!-(k+1)(n-k-2)!,\]
which is larger than \((n-k)!\) if \(k \geq 4\) and \(n \leq 2k\). However, Deza and Frankl conjectured in \cite{deza-frankl} that if \(n\) is large enough depending on \(k\), the \(k\)-cosets {\em are} the largest \(k\)-intersecting subsets of \(S_n\):

\begin{conj}[Deza-Frankl]
\label{conj:deza-frankl}
For any \(k \in \mathbb{N}\) and any \(n\) sufficiently large depending on \(k\), if \(I \subset S_n\) is \(k\)-intersecting, then \(|I| \leq (n-k)!\). Equality holds if and only if \(I\) is a \(k\)-coset of \(S_n\).
\end{conj}

This is our main result. Our proof uses eigenvalue techniques, together with the representation theory of \(S_n\). In fact, it was proved by the first author and the last two authors independently in 2008. Our two proofs of the upper bound are essentially equivalent, hence the joint paper. However, the latter two authors proved the equality statement directly, via their generalization of Birkhoff's Theorem (Theorem \ref{thm:genbirkhoff}), whereas the first author deduced it from the following `stability' result, proved in \cite{kstability}:

\begin{theorem}[\cite{kstability}]
\label{thm:ellis-stability}
Let \(k \in \mathbb{N}\) be fixed, let \(n\) be sufficiently large depending on \(k\), and let \(I \subset S_n\) be a \(k\)-intersecting family which is not contained within a \(k\) coset. Then \(I\) is no larger than the family
\begin{eqnarray*}
&&\{\sigma \in S_n:\ \sigma(i)=i \ \ \forall i \leq k,
\ \sigma(j)=j\ \textrm{for some}\ j > k+1\}\\
&\cup & \{(1,\ k+1),(2,\ k+1),\ldots,(k, \ k+1)\},
\end{eqnarray*}
which has size \((1-1/e+o(1))(n-k)!\).
\end{theorem} This may be seen as an analogue of the Hilton-Milner theorem \cite{hiltonmilner} on intersecting families of $r$-sets; the \(k=1\) case was conjectured by Cameron and Ku in \cite{cameron-ku}.

When there exists a sharply \(k\)-transitive subset of \(S_n\), the `partitioning' argument in the above proof of Theorem \ref{one-intersecting} can be replaced by a Katona-type averaging argument, proving Conjecture \ref{conj:deza-frankl} in this case.

[Recall that a subset $T \subset S_{n}$ is said to be $k$-\textit{transitive} if for any distinct $i_{1},\ldots,i_{k} \in [n]$
and any distinct $j_{1},\ldots,j_{k} \in [n]$, there exists $\sigma \in T$ such that $\sigma(i_{t}) = j_{t}$ for each \(t \in [k]\); \(T\) it is said to be \textit{sharply} $k$-\textit{transitive} if there exists a unique such $\sigma \in T$. Note that a $k$-transitive subset $T \subset S_{n}$ is sharply $k$-transitive if and only if it has size $n(n-1) \ldots (n-k+1)$.]

Indeed, suppose that $S_{n}$ has a sharply $k$-transitive subset $T$. Then any left translate $\sigma T$ of $T$ is also sharply $k$-transitive, so any two distinct permutations
in $\sigma T$ agree in at most $k-1$ places. Let \(I \subset S_n\) be \(k\)-intersecting; then $|I \cap \sigma T| \leq 1$ for each \(\sigma \in S_n\). Averaging over all \(\sigma \in S_n\) gives $|I| \leq (n-k)!$.

For $k=2$ and $n=q$ a prime power, $S_{n}$ has a sharply 2-transitive \textit{subgroup}:
identify the ground set with the finite field $\mathbb{F}_{q}$ of order $q$, and take $H$ to be the group
of all affine maps $x \mapsto ax + b$ ($a \in \mathbb{F}_{q}\setminus \{0\},\ b \in \mathbb{F}_{q}$).
Any two distinct permutations in $H$ agree in at most 1 point, and the same is true for any left coset of \(H\). So if \(I \subset S_n\) is 2-intersecting, then \(I\) contains at most one permutation from each left coset of \(H\). Since the \((n-2)!\) left cosets of $H$ partition $S_{n}$, this implies that $|I| \leq (n-2)!$.

Similarly, for $k = 3$ and $n = q+1$ (where $q$ is a prime power), $S_{n}$ has a sharply 3-transitive subgroup:
identify the ground set with $\mathbb{F}_{q} \cup \{\infty\}$, and take $H$ to be the group of all M\"obius
transformations \[x \mapsto \frac{ax+b}{cx+d} \quad (a,b,c,d \in \mathbb{F}_{q},\ ad-bc \neq 0).\]
However, it is a classical result of C. Jordan \cite{jordan} that the only sharply $k$-transitive permutation groups for
 $k \geq 4$ are $S_{k}$ (for $k \geq 4$), $A_{k-2}$ (for $k \geq 8$), $M_{11}$ (for $k=4$) and
 $M_{12}$ (for $k = 5$), where $M_{11},M_{12}$ are the Matthieu groups.
 Moreover, sharply $k$-transitive subsets of $S_{n}$ have not been found for any other values of $n$ and $k$.
 Thus, it seems unlikely that this approach can work in general. Instead, we will use a different approach.

Recall that if \(G\) is a group and \(X \subset G\) is inverse-closed, the {\em Cayley graph on} \(G\) {\em generated by} \(X\) is the graph with vertex-set \(G\) and edge-set \(\{\{g,h\}:\ g h^{-1} \in X\}\). Let \(\Gamma_1\) be the Cayley graph on \(S_n\) generated by the set of {\em fixed-point-free} permutations,
\[\textrm{FPF} = \{\sigma \in S_n:\ \sigma(i) \neq i\ \forall i \in [n]\}.\]
Note that a 1-intersecting subset of \(S_n\) is precisely an independent set in \(\Gamma_1\). It turns out that calculating the least eigenvalue of \(\Gamma_1\) (meaning the least eigenvalue of its adjacency matrix) and applying Hoffman's bound (Theorem \ref{thm:hoffman}) yields an alternative proof of Theorem \ref{one-intersecting}. Calculating the least eigenvalue of \(\Gamma_1\) is non-trivial, requiring use of the representation theory of \(S_n\). It was first done by Renteln \cite{renteln}, using symmetric functions, and independently and slightly later by Friedgut and Pilpel \cite{friedguttalk}, and also by Godsil and Meagher \cite{godsil-2008}. As observed in \cite{friedguttalk} and \cite{godsil-2008}, this also leads to an alternative proof that the \(1\)-cosets are the unique largest 1-intersecting subsets.

The obvious generalization of this approach fails for \(k\)-intersecting subsets of \(S_n\). Let \(\Gamma_{k}\) be the Cayley graph on \(S_n\) generated by the set
\[\textrm{FPF}_k = \{\sigma \in S_n:\ \sigma \textrm{ has less than }k\textrm{ fixed points}\}.\]
A \(k\)-intersecting subset of \(S_n\) is precisely an independent set in \(\Gamma_k\), so our task is to find the largest independent sets in \(\Gamma_k\). Unfortunately, for \(k\) fixed and \(n\) large, calculating the least eigenvalue of \(\Gamma_k\) and applying Hoffman's bound only gives an upper bound of \(\Theta((n-1)!)\) on the size of a \(k\)-intersecting family.

A key idea of our proof is to choose various subgraphs of \(\Gamma_k\), and to construct a `pseudo-adjacency matrix' \(A\) for \(\Gamma_k\) which is a suitable real linear combination of the adjacency matrices of these subgraphs. We then apply a weighted version of Hoffman's Theorem (Theorem \ref{thm:hoffman-linear}) to this linear combination, in order to prove the upper bound in Conjecture \ref{conj:deza-frankl}. The subgraphs chosen will be Cayley graphs generated by various unions of conjugacy-classes of \(S_n\); this will enable use to calculate their eigenvalues using the representation theory of \(S_n\). Most of the work of the proof is in showing that an appropriate linear combination exists.

\subsection{Our main results}
A remark on terminology: we will often identify a subset $A$ of $S_n$, its
characteristic function $1_A : S_n \rightarrow \{0,1\}$, and its
characteristic vector $v_A \in \mathbb{R}^{|S_n|}$; so when we say
that a set $A$ is spanned by sets $B_1, \ldots, B_t$, we mean that
$v_A \in \Span\{v_{B_1}, \ldots, v_{B_t} \}$.

Our main results in this paper are as follows.
\begin{theorem} \label{thm:final-result}
  For any \(k \in \mathbb{N}\), and any \(n\) sufficiently large depending on \(k\), if \(I \subset S_n\) is \(k\)-intersecting, then \(|I| \leq (n-k)!\). Equality holds if and only if \(I\) is a \(k\)-coset of \(S_n\).
\end{theorem}
We also prove a cross-intersecting version of this theorem:
\begin{definition}
  Two sets $I,J \subset S_n$ are {\em $k$-cross-intersecting} if
  every permutation in \(I\) \(k\)-intersects every permutation in \(J\).
\end{definition}

\begin{theorem}\label{thm:X-main}
 For any $k \in \mathbb{N}$ and any \(n\) sufficiently large depending on \(k\), if $I, J \subset S_n$
are $k$-cross-intersecting, then $|I||J| \leq ((n-k)!)^2$. Equality holds if and only if $I=J$ and $I$ is a $k$-coset of \(S_n\).
\end{theorem}
The \(k=1\) case of the above was a conjecture of Leader \cite{leader-private-2008}.

Our argument proceeds in the following steps. (In order to
not disrupt the flow of the paper, some of the
representation-theoretic terms used will only be defined later.)

First, we bound the size of a $k$-intersecting family.
\begin{theorem} \label{main-theorem}
  For any $k \in \mathbb{N}$ and any \(n\) sufficiently large depending on $k$, if $I \subset S_n$
  is $k$-intersecting, then $|I| \leq (n-k)!$. Moreover, if $I,J \subset S_n$ are $k$-cross-intersecting, then
 $|I||J| \leq ((n-k)!)^{2}$.
\end{theorem}
Next, we describe the Fourier transform of the characteristic functions of the families which
achieve this bound.  Let $V_k$ be the linear subspace of real-valued functions on
$S_n$ whose Fourier transform is supported on irreducible
representations corresponding to partitions $\mu$ of $n$ such that
$\mu \geq (n-k,1^k)$, where \(\geq\) denotes the lexicographic order (see section \ref{sec:rep-sn}).
\begin{theorem} \label{thm:span-is-maximal}
  For $k$ fixed and \(n\) sufficiently large depending on \(k\), if $I \subset S_n$ is
  a $k$-intersecting family of size $(n-k)!$, then $1_I \in V_k$.
  Furthermore if $I,J\subset S_n$ are $k$-cross-intersecting and $|I||J|=(n-k)!^2$,
  then $1_I, 1_J \in V_k$.
\end{theorem}
We then prove the following:
\begin{theorem} \label{thm:cosets-span}
  $V_k$ is spanned by the characteristic functions of the $k$-cosets
  of $S_n$.
\end{theorem}
Finally, we complete the proof of our main theorem using the following
combinatorial result:
\begin{theorem} \label{thm:span-is-disjoint-union}
  For any $k \in \mathbb{N}$, if \(f\) is a Boolean function on \(S_n\) which is spanned by the $k$-cosets, then $f$ is the characteristic function of a disjoint
  union of $k$-cosets.
\end{theorem}
Clearly, the last four theorems immediately imply our main result, Theorem \ref{thm:final-result},
and its cross-intersecting version, Theorem \ref{thm:X-main}.

\subsection{Structure of the paper}
In section \ref{sec:background} we provide the background that we will
use from general representation theory and graph theory. In section
\ref{sec:rep-sn} we prove all necessary results and lemmas that pertain to
representation theory of $S_n$. Section \ref{sec:main-proofs} ties together the
results of the previous two sections in order to bound the size and
provide a Fourier characterization of the largest $k$-intersecting families.
Finally, in section \ref{sec:boolean-functions} we show that this
characterization holds only for $k$-cosets.

\section{Background} \label{sec:background}
\subsection{General representation theory}
In this section, we recall the basic notions and results we need from general representation theory. For more background, the reader may consult \cite{isaacs}.

Let \(G\) be a finite group, and let \(F\) be a field. A {\em representation of \(G\) over \(F\)} is a pair \((\rho,V)\), where \(V\) is a finite-dimensional vector space over \(F\), and \(\rho:\ G \to GL(V)\) is a group homomorphism from \(G\) to the group of all invertible linear endomorphisms of \(V\). The vector space \(V\), together with the linear action of \(G\) defined by \(gv = \rho(g)(v)\), is sometimes called an \(FG\)-{\em module}.  A {\em homomorphism} between two representations \((\rho,V)\) and \((\rho',V')\) is a linear map \(\phi:V \to V'\) such that such that \(\phi(\rho(g)(v)) = \rho'(g)(\phi(v))\) for all \(g \in G\) and \(v \in V\). If \(\phi\) is a linear isomorphism, the two representations are said to be {\em equivalent}, and we write \((\rho,V) \cong (\rho',V')\). If \(\dim(V)=n\), we say that \(\rho\) {\em has dimension} \(n\). If \(V = F^{n}\), then we call \(\rho\) a {\em matrix representation}; choosing an \(F\)-basis for a general \(V\), one sees that any representation is equivalent to some matrix representation.

The representation \((\rho,V)\) is said to be {\em irreducible} if it has no proper subrepresentation, i.e. there is no proper subspace of \(V\) which is \(\rho(g)\)-invariant for all \(g \in G\).

The {\em group algebra} \(F[G]\) denotes the \(F\)-vector space with basis \(G\) and multiplication defined by extending the group multiplication linearly. In other words,
\[F[G] = \left\{\sum_{g \in G}x_{g}g:\ x_{g} \in F\ \forall g \in G\right\},\]
and
\[\left(\sum_{g \in G} x_{g}g\right)\left(\sum_{h\in G}y_{h}h\right) = \sum_{g,h \in G} x_{g}y_{h} (g h).\]
Idenfifying \(\sum_{g \in G} x_g g\) with the function \(g \mapsto x_g\), we can view the vector space \(F[G]\) as the space of all \(F\)-valued functions on \(G\). The representation defined by
\[\rho(g)(x) = gx\quad (g \in G,\ x \in F[G])\]
is called the {\em left regular representation} of \(G\); the corresponding \(FG\)-module is called the {\em group module}. This will be useful when we describe irreducible representations of \(S_n\).

When \(F = \mathbb{R}\) or \(\mathbb{C}\), it turns out that there are only finitely many equivalence classes of irreducible representations of \(G\), and {\em any} representation of \(G\) is isomorphic to a direct sum of irreducible representations of \(G\). Hence, we may choose a set of representatives \(R\) for the equivalence classes of complex irreducible representations of \(G\). For the rest of section \ref{sec:background}, \(R\) will be fixed, and will consist of matrix representations.

If \((\rho,V)\) is a representation of \(G\), and \(\alpha\) is a linear endomorphism of \(V\), we say that \(\alpha\) {\em commutes} with \(\rho\) if \(\alpha \circ (\rho(g)) = \rho(g) \circ \alpha\) for every \(g \in G\). (So an isomorphism of \((\rho,V)\) is simply an invertible linear endomorphism which commutes with \(\rho\).) We will make use of the following:

\begin{lemma}[Schur's Lemma]
If \(G\) is a finite group, and \((\rho,V)\) is a complex irreducible representation of \(G\), then the only linear endomorphisms of \(V\) which commute with \(\rho\) are scalar multiples of the identity.
\end{lemma}

If \(F=\mathbb{R}\) or \(\mathbb{C}\), we may define an inner product on \(F[G]\) as follows:
\[\langle \phi,\psi \rangle = \frac{1}{|G|} \sum_{g \in G} \phi(g) \overline{\psi(g)}.\]
If \((\rho,V)\) is a complex representation of \(V\), the {\em character} \(\chi_{\rho}\) of \(\rho\) is the map defined by
\begin{eqnarray*}
\chi_{\rho}:  G & \to & \mathbb{C};\\
 \chi_{\rho}(g) &=& \textrm{Tr} (\rho(g)),
\end{eqnarray*}
where \(\textrm{Tr}(\alpha)\) denotes the trace of the linear map \(\alpha\) (i.e. the trace of any matrix of \(\alpha\)). Note that \(\chi_{\rho}(\textrm{Id}) = \dim(\rho)\), and that \(\chi_{\rho}\) is a {\em class function} on \(G\) (meaning that it is constant on each conjugacy-class of \(G\).)

The usefulness of characters lies in the following
\begin{fact}
Two complex representations are isomorphic if and only if they have the same character; the set of complex irreducible characters is an orthonormal basis for the space of class functions in \(\mathbb{C}[G]\).
\end{fact}

If \(\rho\) is any complex representation of \(G\), its character satisfies \(\chi_{\rho}(g^{-1}) = \overline{\chi_{\rho}(g)}\) for every \(g \in G\). In our case, \(G = S_n\), so \(g^{-1}\) is conjugate to \(g\) for every \(g\), and therefore \(\overline{\chi_{\rho}(g)} = \chi_{\rho}(g)\), i.e. all the characters are real-valued. The irreducible characters of \(S_n\) are therefore an orthonormal basis for the space of class functions on \(\mathbb{R}[S_n]\).

Given two representations \((\rho,V)\) and \((\rho',V')\) of \(G\), we can form their direct sum, the representation \((\rho \oplus \rho',V \oplus V')\), and their tensor product, the representation \((\rho \otimes \rho',V \otimes V')\). We have \(\chi_{\rho \oplus \rho'} = \chi_{\rho}+ \chi_{\rho'}\), and \(\chi_{\rho \otimes \rho'} = \chi_{\rho} \cdot \chi_{\rho'}\).

\subsection{Fourier transforms and convolutions}
In this section, we recall the basic notions of Fourier analysis on finite non-Abelian groups. For more background, see for example \cite{terras}. Noting that the normalization chosen differs
in various texts, we set out our convention below.
\begin{definition}
Let \(G\) be a finite group, and let $f,g : G \ra \mathbb{R}$ be two real-valued functions on $G$. Their {\em
convolution} $f*g$ is the real-valued function defined by
\begin{equation}
f*g (t) = \frac {1}{|G|} \sum_{s \in G} f(ts^{-1})g(s)\quad (t \in G).
\end{equation}
\end{definition}

\begin{definition}
The {\em Fourier transform} of a real-valued function $f : G \to \mathbb{R}$ is a matrix-valued function on irreducible representations; its value at the irreducible representation $\rho$ is the matrix
\begin{equation}
 \widehat{f}(\rho) = \frac {1}{|G|} \sum_{s \in G} f(s)\rho(s).
\end{equation}
\end{definition}
We now recall two related formulas we will need: the
Fourier transform of a convolution, and the Fourier inversion formula.
If $f,g : G \ra \mathbb{R}$, and \(\rho\) is an irreducible representation of \(G\), then
\begin{equation}
  \widehat{f*g}(\rho)= \widehat{f}(\rho)\widehat{g}(\rho).
\end{equation}
The Fourier transform is invertible; we have:
\begin{equation}
    f(s) = \sum_{\rho \in R} \dim(\rho) \tr
    \left[ \widehat{f}(\rho) \rho(s^{-1})\right].
\end{equation}
In other words, the Fourier transform contains all the information about a function \(f\).

\subsection{Cayley Graphs}
\label{subsec:cayleygraphs}
Recall that if \(G\) is a group, and \(X \subset G\) is inverse-closed, the {\em Cayley graph on \(G\) generated by \(X\)} is the graph with vertex-set \(G\) and edge-set $\{ \{u,v\} \in G^{(2)} :\ uv^{-1} \in
X\}$; it is sometimes denoted by \(\Cay(G,X)\). In fact, we will only be considering
cases where \(G=S_n\) and $X$ is a union of conjugacy classes (i.e., $X$ is
conjugation-invariant).

The relevance of this notion to our problem stems from the following
observation. Consider the Cayley graph \(\Gamma_1\) on $S_n$ generated by
the set of {\em fixed-point free} permutations,
$$\fpf = \{ \sigma \in S_{n}:\ \sigma(i) \neq i\ \forall i \in [n]\}.$$
As observed in section \ref{sec:history}, a 1-intersecting family of permutations
is precisely an independent set in this graph. More generally, a
$k$-intersecting family of permutations is precisely an independent
set in the Cayley graph \(\Gamma_{k}\) on $S_n$ with generating set
$$\fpf_k = \{\sigma \in S_n:\ \sigma \textrm{ has at most }k\textrm{ fixed points}\}.$$

For any real matrix \(A \in \mathbb{R}[G \times G]\), the {\em left action} of \(A\) on \(\mathbb{R}[G]\) is defined as follows:
\[(A f)(\sigma) = \sum_{\tau \in G} A_{\sigma,\tau} f(\tau).\]

The main observation of this subsection is that the adjacency matrix of a
Cayley graph operates on functions in \(\mathbb{R}[G]\) by convolution with the characteristic
function of the generating set.

\begin{theorem} \label{product-is-convolution}
  Let \(G\) be a finite group, let \(X \subset G\) be inverse-closed, let $\Cay(G,X)$ be the Cayley graph on \(G\) generated by \(X\), and let $A$ be the adjacency matrix of \(\Cay(G,X)\). Then for any function $f : G \rightarrow \mathbb{R}$,
  \begin{equation}
\label{eq:conv}
    Af = |G| (1_X * f).
  \end{equation}
\end{theorem}
\begin{proof}
  For any \(f \in \mathbb{R}[G]\), and any \(\sigma \in G\), we have
  \[
  (Af)(\sigma) = \sum_{\tau \in G} A_{\sigma, \tau} f(\tau) = \sum_{\tau \in G}
  1_X(\sigma \tau^{-1}) f(\tau) = |G|(1_X * f)(\sigma),
  \]
  as required.
\end{proof}

Taking the Fourier transform of both sides of (\ref{eq:conv}), we obtain:
\begin{equation} \label{fourier-of-adjacency}
  \widehat{Af} = |G| \cdot \widehat{1_X} \widehat{f}.
\end{equation}

We will see shortly that if \(X\) is conjugation-invariant, as in our case, then $\widehat{1_X}$ is a scalar
matrix, so (\ref{fourier-of-adjacency}) essentially reveals
all eigenfunctions of the operator $A$: it is well
known that for any finite group \(G\), the entries of the matrices of a complete set of complex irreducible
representations of \(G\) form an orthogonal basis for the space of all complex-valued functions on \(G\). So in our case, these are also a
complete set of eigenfunctions of $A$.
\begin{theorem}[Schur; Babai; Diaconis, Shahshahani; Roichman; others \cite{babai, diaconis-1980, roichman-1994}] \label{thm:cayley-spectrum}
  Let \(G\) be a finite group, and let \(R\) be a complete set of complex irreducible matrix representations of \(G\), as above. Let $X \subset G$ be inverse-closed and conjugation-invariant, and let $\Cay(G,X)$ be the Cayley graph on $G$ with generating set $X$. Let $A$ be the adjacency matrix of \(\Cay(G,X)\). For any $\rho \in R$, and any \(i,j \leq \dim(\rho)\),
  consider the function $\rho_{ij}(\sigma) = \rho(\sigma)_{ij}$. Then $\{
  \rho_{ij} \}_{\rho,i,j}$ is a complete set of eigenvectors of $A$.  Furthermore, the eigenvalue of $\rho_{ij}$
  is
  \begin{equation} \label{cayley-eigenvalue}
    \lambda_{\rho} = \frac{1}{\dim(\rho)} \sum_{\tau \in X} \chi_{\rho}(\tau) =  \frac{|G|\langle \chi_{\rho},1_X \rangle}{\dim(\rho)},
  \end{equation}
  which depends only on $\rho$.
\end{theorem}
\begin{proof}
   Note that, due to (\ref{fourier-of-adjacency}),
  the claim regarding eigenvectors will follow immediately once we have shown
  that $\widehat{1_X}(\rho)$ is a scalar matrix for every irreducible representation $\rho$.
  To do this, we will show that $ \widehat{1_X}(\rho)$ commutes with every $\rho(\sigma)$, which will imply
  the result (by Schur's Lemma.) Indeed, for every \(\sigma \in G\),
  \begin{eqnarray*}
    \rho(\sigma)^{-1} \widehat{1_X}(\rho) \rho(\sigma)
    & = & \rho(\sigma^{-1}) \widehat{1_X}(\rho)\rho(\sigma) \\
    & = & \frac{1}{|G|}\sum_{\tau \in G} \rho(\sigma^{-1}) \rho(\tau) 1_X(\tau) \rho(\sigma) \\
    & = & \frac{1}{|G|}\sum_{\tau \in X}  \rho(\sigma^{-1} \tau \sigma) \\
    &=& \frac{1}{|G|}\sum_{\tau \in X}  \rho( \tau )  \\
    &=& \widehat{1_X}(\rho),
  \end{eqnarray*}
where we have used the fact that \(X\) is conjugation-invariant for the fourth equality. Hence, by Schur's Lemma, $\widehat{1_X}(\rho)$ is indeed a scalar matrix; write \(\widehat{1_X}(\rho) = c_{\rho}I\). To calculate $c_{\rho}$, note that for any $i \le \dim(\rho)$,
  $$
    c_{\rho} = \frac{1}{|G|}\sum_{\tau \in X} 1_X(\tau) \rho_{i,i}(\tau).
  $$
  Summing over $i$, we obtain:
  $$
    c_{\rho}\dim(\rho)
    = \frac{1}{|G|}\sum_{\tau \in X}  \tr[\rho(\tau)]
    = \frac{1}{|G|} \sum_{\tau \in X} \chi_{\rho}(\tau).
  $$
  From (\ref{fourier-of-adjacency}), $\lambda_\rho= |G|c_\rho $, completing the proof.
\end{proof}

\subsection{Hoffman's bound}
First, a word regarding normalization, as this is always a potential source of confusion
when doing Fourier analysis. Given a graph $G$, we use the uniform probability measure on the vertex-set \(V\) of \(G\), {\em not} the counting measure. The uniform measure induces the following inner product on \(\mathbb{R}[V]\):
\[\langle f,g \rangle = \frac{1}{|V|} \sum_{v \in V}f(v)g(v);\]
this induces the Euclean norm
\[||f||_2 = \sqrt{\langle f,f \rangle}.\]
If \(G = (V,E)\) is a graph, the {\em adjacency matrix} \(A\) of \(G\) is defined by
\[A_{v,w} = 1_{\{vw \in E(G)\}}\quad (v,w \in V(G)).\]
This is a real, symmetric, \(n \times n\) matrix, so there exists an orthonormal system of \(n\) eigenvectors of \(A\), which forms a basis for \(\mathbb{R}[V]\). (Note that the eigenvalues of \(A\) are often referred to as the eigenvalues of \(G\).)

Hoffman \cite{hoffman} observed the following useful bound on the measure of an independent set in a
regular graph, in terms of the eigenvalues of the graph:

\begin{theorem} \label{thm:hoffman} \cite{hoffman}
  Let $G=(V,E)$ be a $d$-regular, \(n\)-vertex graph. Let \(A\) be the adjacency matrix of \(G\). Let \(\{v_1,v_2,\ldots,v_n\}\) be an orthonormal system of eigenvectors of \(A\), with corresponding eigenvalues \(d=\lambda_1 \geq \lambda_2 \geq \ldots \geq \lambda_n = \lambda_{\min}\) (so that \(v_1\) is the all-1's vector). If $I \subset V$ is an independent set in $G$, then
  \begin{equation} \label{hoffman-formula}
    \frac{|I|}{|V|} \leq \frac{-\lambda_{\min}}{\lambda_1 - \lambda_{\min}}.
  \end{equation}
  If equality holds, then
  $$
    1_I \in
    \Span \left( \{ v_1 \} \cup \{ v_i : \lambda_i = \lambda_{\min} \} \right).
  $$
\end{theorem}
\begin{proof}
  Let $f = 1_I$, and let $\alpha =
  \frac{|I|}{|V|}$. Observe that
\[f^t A f = \sum_{v,w \in I} A_{v,w} = 2e(G[I]) = 0,\]
since \(I\) is independent. Write $f$ as a linear combination of the eigenvectors:
  $$
    f = \sum_{i=1}^{n} a_i v_i.
  $$
  Then $\alpha = \langle f, v_1 \rangle = a_1$. Moreover, by
  Parseval's identity, we have $\sum_i a_i^2 = ||f||_2^2 =
  \alpha$.  Now,
  $$
    0
    = f^t A f
    = \sum_{i=1}^{n} a_i^2 \lambda_i
    \geq a_1^2 \lambda_1 + \sum_{i=2}^{n} a_i^2 \lambda_{\min}
    = \alpha^2 \lambda_1 + (\alpha - \alpha^2) \lambda_{\min}.
  $$
  Rearranging gives (\ref{hoffman-formula}). If equality holds, then \(a_{i} \neq 0\) implies that \(i = 1\) or \(\lambda_{i} = \lambda_{\min}\), completing the proof.
\end{proof}

A variant of Hoffman's theorem, which will be crucial for us,
comes from weighting the edges of the graph \(G\) with real (possibly negative) weights. 

\begin{theorem} \label{thm:hoffman-linear}
  Let $G=(V,E)$ be an \(n\)-vertex graph. Let $G_1, \ldots, G_t$ be regular, spanning subgraphs of \(G\), all having \(\{v_1,v_2\ldots,v_n\}\) as an orthonormal system of eigenvectors (where \(v_1\) is the all-1's vector). Let $\lambda_i^{(j)}$ be the eigenvalue of $v_i$ in
  $G_j$. Let $\beta_1, \ldots, \beta_t \in \mathbb{R}$, and let
  $\lambda_i = \sum_j \beta_j \lambda_i^{(j)}$, and let \(\lambda_{\min} = \min_{i} \lambda_i\). If $I \subset V$ is an independent set in $G$, then
  \begin{equation}
\label{eq:weightedhoffman}
\frac{|I|}{|V|} \leq \frac{-\lambda_{\min}}{\lambda_1 - \lambda_{\min}}.
  \end{equation}
  If equality holds, then
  $$
    1_I \in
    \Span \left( \{ v_1 \} \cup \{ v_i : \lambda_i = \lambda_{\min} \} \right).
  $$
\end{theorem}
\begin{proof}
  The proof is a simple generalization of that of Theorem
  \ref{thm:hoffman}. For each $j$, let $A_j$ be the adjacency matrix of $G_j$, and
  let $A = \sum_j \beta_j A_j$. We have
 \begin{eqnarray*}
    0
    & = &  f^t A f
    = \sum_{j=1}^{t} \beta_j f^t A_j f
    = \sum_{j=1}^{t} \beta_j \sum_{i=1}^{n} a_i^2 \lambda_i^{(j)}
    = \sum_{i=1}^{n} a_i^2 \lambda_i  \\
    & \geq &   \lambda_1 a_1^2 + \sum_i a_i^2 \lambda_{\min}
    = \alpha^2 \lambda_1 + (\alpha - \alpha^2) \lambda_{\min}.
  \end{eqnarray*}

  Rearranging gives (\ref{eq:weightedhoffman}). If equality holds, then \(a_{i} \neq 0\) implies that \(i = 1\) or \(\lambda_i = \lambda_{\min}\), completing the proof.
\end{proof}

We may think of the \(\lambda_i\)'s above as the eigenvalues of the linear combination of graphs
\[Y = \sum_{j=1}^{t} \beta_j G_j.\]
The corresponding linear combination of adjacency matrices
\[A = \sum_{j=1}^{t} \beta_j A_j\]
is a so-called {\em pseudo-adjacency matrix} for \(G\), meaning a symmetric matrix such that \(A_{v,w}  =0\) whenever \(vw \notin E(G)\); the \(\lambda_i\)'s are the eigenvalues of \(A\).

Finally, we will need the following cross-independent version of
Hoffman's Theorem. Variants of this theorem can be found in several
sources, e.g. \cite{AKKMS}.
\begin{theorem} \label{thm:Xhoffman}
  Let $G=(V,E)$ be a $d$-regular, \(n\)-vertex graph, and let \(\{v_1,v_2,\ldots,v_n\}\) be an orthonormal system of eigenvectors of \(G\), with corresponding eigenvalues \(d=\lambda_1,\lambda_2,\ldots,\lambda_n\) ordered by descending absolute value (so that \(v_1\) is again the all-1's vector). Let $I,J \subset V$ be (not necessarily disjoint) sets of
  vertices in $G$ such that there are no edges of \(G\) between $I$
  and $J$. Then
  \begin{equation} \label{Xhoffman-formula}
   \sqrt{ \frac{|I|}{|V|}\cdot \frac{|J|}{|V|}} \leq
   \frac{|\lambda_2|}{\lambda_1 + |\lambda_2|}.
  \end{equation}
  If equality holds, then
  $$
    1_I, 1_J \in \Span \left( \{ v_1 \} \cup \{ v_i : |\lambda_i| =
    |\lambda_2| \} \right).
  $$
\end{theorem}
\begin{proof}
  Let $f$ and $g$ be the characteristic functions of $I$ and $J$ respectively. As in the proof of Hoffman's Theorem,
  write
  $$
    f = \sum_{i=1}^{n} a_i v_i,\  g = \sum_{i=1}^{n} b_i v_i.
  $$
  Let $\alpha = \frac{|I|}{|V|},\beta = \frac{|J|}{|V|}$. We have
  \begin{equation} \label{Xhoffman-equation}
    0 = 2e(I,J)
    = f^t A g
    = \sum_{i=1}^{n} a_ib_i \lambda_i
    = \alpha \beta \lambda_1 +\sum_{i=2}^{n} a_ib_i \lambda_i.
  \end{equation}
  Hence, by the Cauchy-Schwarz inequality,
  $$
    \alpha \beta \lambda_1 =
    \left|\sum_{i=2}^{n} a_ib_i \lambda_i\right| \leq
    \sum_{i=2}^n | a_ib_i \lambda_i| \leq
    |\lambda_2| \sqrt{\alpha -\alpha^2} \sqrt{\beta - \beta^2}.
  $$
  Rearranging gives
  $$
    \sqrt{\frac{\alpha \beta}{(1-\alpha)(1-\beta)}} \leq \frac{|\lambda_2|}{\lambda_1}.
  $$
 Recall the AM-GM inequality for two real variables (see for example \cite{amgm}):
$$
(\alpha + \beta)/2 \geq \sqrt{\alpha \beta}.
$$
This implies that
$$
(1-\alpha)(1-\beta) = 1-\alpha-\beta+\alpha \beta \leq 1-2\sqrt{\alpha \beta}+\alpha\beta = (1-\sqrt{\alpha \beta})^2,
$$
and therefore
$$
\sqrt{(1-\alpha)(1-\beta)} \leq 1-\sqrt{\alpha\beta}.
$$
Hence,
$$
    \frac{\sqrt{\alpha \beta}}{1-\sqrt{\alpha\beta}} \leq \frac{|\lambda_2|}{\lambda_1}.
$$
Rearranging gives
$$
\sqrt{\alpha \beta} \leq \frac{|\lambda_2|}{\lambda_1 + |\lambda_2|},
$$
as required. The case of equality is dealt with as in the original Hoffman theorem.
\end{proof}

Note that the generalization of Hoffman's theorem that we mention
above holds, with the obvious modifications, in the cross-independent case.

\section{Representation Theory of $S_n$}
\label{sec:rep-sn}
In this section we gather all the necessary background and results regarding the representation theory of $S_n$. Readers
familiar with the basics of this theory are invited to skip the
following subsection.
\subsection{Basics}
Our treatment follows Sagan \cite{sagan}.

A \emph{partition} of \(n\) is a non-increasing sequence of integers summing to \(n\), i.e. a sequence $\lambda = (\lambda_1, \ldots, \lambda_k)$ with \(\lambda_{1} \geq \lambda_{2} \geq \ldots \geq \lambda_{k}\) and \(\sum_{i=1}^{k} \lambda_{i}=n\); we write \(\lambda \vdash n\). For example, \((3,2,2) \vdash 7\); we sometimes use the shorthand \((3,2,2) = (3,2^{2})\). The following two orders on partitions of \(n\) will be useful.

\begin{definition}
  (Dominance order) Let $\lambda = (\lambda_1, \ldots, \lambda_r)$ and $\mu = (\mu_1,
  \ldots, \mu_s)$ be partitions of $n$. We say that $\lambda \domgeq
  \mu$ ($\lambda$ \emph{dominates} $\mu$) if $\sum_{j=1}^{i} \lambda_i \geq \sum_{j=1}^{i} \mu_i$ for all \(i\) (where we define \(\lambda_{i} = 0 \) for all \(i > r\), and \(\mu_{j} = 0\) for all \(j > s\)).
\end{definition}

It is easy to see that this is a partial order.

\begin{definition}
  (Lexicographic order) Let $\lambda = (\lambda_1, \ldots, \lambda_r)$ and $\mu = (\mu_1,
  \ldots, \mu_s)$ be partitions of $n$. We say that $\lambda > \mu$ if \(\lambda_{j} > \mu_{j}\), where \(j = \min\{i \in [n]:\ \lambda_{i} \neq \mu_{i}\}\).
\end{definition}
It is easy to see that this is a total order which extends the dominance order.

The \emph{cycle-type} of a permutation \(\sigma \in S_{n}\) is the partition of \(n\) obtained by expressing \(\sigma\) as a product of disjoint cycles and listing its cycle-lengths in non-increasing order. The conjugacy-classes of \(S_{n}\) are precisely
\[\{\sigma \in S_{n}: \textrm{ cycle-type}(\sigma) = \lambda\}_{\lambda \vdash n}.\]
Moreover, there is an explicit one-to-one correspondence between irreducible representations of \(S_{n}\) (up to isomorphism) and partitions of \(n\), which we now describe.

  Let \(\lambda = (\lambda_{1}, \ldots, \lambda_{k})\) be a partiton of \(n\). The \emph{Young diagram} of $\lambda$ is
  an array of $n$ boxes, or {\em cells}, having $k$ left-justified rows, where row $i$
  contains $\lambda_i$ cells. For example, the Young diagram of the partition \((3,2^{2})\) is:
  \[
\yng(3,2,2)
\]

If the array contains the numbers \(\{1,2,\ldots,n\}\) inside the cells, we call it a \(\lambda\)-\emph{tableau}, or a tableau of shape \(\lambda\); for example:
\[
\young(617,54,32)
\]
is a \((3,2^{2})\)-tableau. Two \(\lambda\)-tableaux are said to be \emph{row-equivalent} if they have the same numbers in each row; if a \(\lambda\)-tableau \(t\) has rows \(R_{1},\ldots,R_{l} \subset [n]\) and columns \(C_{1},\ldots,C_{l} \subset [n]\), then we let \(R_{t} = S_{R_{1}} \times S_{R_{2}} \times \ldots \times S_{R_{l}}\) be the row-stablizer of \(t\) and \(C_t = S_{C_{1}} \times S_{C_{2}} \times \ldots \times S_{C_{k}}\) be the column-stabilizer.

A \(\lambda\)-\emph{tabloid} is a \(\lambda\)-tableau with unordered row entries (or formally, a row-equivalence class of \(\lambda\)-tableaux); given a tableau \(t\), we write \([t]\) for the tabloid it produces. For example, the \((3,2^{2})\)-tableau above produces the following \((3,2^{2})\)-tabloid:\\
\\
\begin{tabular}{ccc} \{1 & 6 & 7\}\\
\{4 &\ 5\} &\\
\{2 &\ 3\} &\\
\end{tabular}\\
\\
\\
Consider the natural permutation action of \(S_{n}\) on the set \(X^{\lambda}\) of all \(\lambda\)-tabloids; let \(M^{\lambda} = \mathbb{R}[X^{\lambda}]\) be the corresponding permutation module, i.e. the real vector space with basis \(X^{\lambda}\) and \(S_{n}\) action given by extending the permutation action linearly. In general, \(M^{\lambda}\) is reducible. However, we can describe a complete set of real irreducible representations, as follows.

If $t$ is a tableau, let $\kappa_t = \sum_{\pi \in C_t}
  \sgn(\pi) \pi$; this is an element of the group module $\mathbb{R}[S_n]$. Let $e_t = \kappa_t \{ t \}$. This is a (\(\pm1\))-linear combination of tabloids, so is an element of \(M^{\lambda}\); we call the \(e_t\)'s {\em polytabloids}.

\begin{definition}
  Let $\mu$ be a partition of $n$. The {\em Specht module}
  $S^{\mu}$ is the submodule of $M^{\mu}$ spanned by the \(\mu\)-polytabloids:
  $$S^{\mu} = \textrm{Span}\{e_t:\ t \textrm{ is a } \mu\textrm{-tabloid}\}.$$
\end{definition}

\begin{theorem}
\label{thm:spechtmodules}
 The Specht modules are a complete set of pairwise inequivalent, irreducible representations of $S_n$.
\end{theorem}

Hence, any irreducible representation \(\rho\) of \(S_{n}\) is isomorphic to some \(S^{\lambda}\); in this case, we say that \(\rho\) has Young diagram \(\lambda\). For example, \(S^{(n)} = M^{(n)}\) is the trivial representation; \(M^{(1^{n})}\) is the left-regular representation, and \(S^{(1^{n})}\) is the sign representation \textit{sgn}.

From now on we will write \([\lambda]\) for the equivalence class of the irreducible representation \(S^{\lambda}\), \(\chi_{\lambda}\) for the character \(\chi_{S^{\lambda}}\), and \(\xi_{\lambda}\) for the character of \(M^{\lambda}\). Notice that the set of \(\lambda\)-tabloids forms a basis for \(M^{\lambda}\), and therefore \(\xi_{\lambda}(\sigma)\), the trace of the corresponding permutation representation, is precisely the number of \(\lambda\)-tableaux fixed by \(\sigma\).

We now explain how the permutation modules \(M^{\mu}\) decompose into irreducibles.

\begin{definition}
Let \(\lambda,\mu\) be partitions of \(n\). A {\em \(\lambda\)-tableau} is produced by placing a number between 1 and \(n\) in each cell of the Young diagram of \(\lambda\); if it has \(\mu_{i}\) \(i\)'s \((1 \leq i \leq n)\) it is said to have {\em content} \(\mu\). A generalized \(\lambda\)-tableau is said to be {\em semistandard} if the numbers are non-decreasing along each row and strictly increasing down each column.
\end{definition}

  \begin{definition} \label{def:kostka}
  Let $\lambda, \mu$ be partitions of $n$. The {\em Kostka number}
  $K_{\lambda,\mu}$ is the number of semistandard generalized $\lambda$-tableaux
  with content $\mu$.
\end{definition}

\begin{theorem} \label{thm:young-rule}
  (Young's rule) Let $\mu$ be a partition of $n$. Then the permutation module
  $M^{\mu}$ decomposes as
  \[M^{\mu} \cong \oplus_{\lambda \vdash n} K_{\lambda, \mu}
  S^{\lambda}.\]
  Hence,
  \[\xi_{\mu}=\sum_{\lambda \vdash n} K_{\lambda,\mu}\chi_{\lambda}.\]
\end{theorem}

For example, \(M^{(n-1,1)}\), which corresponds to the natural permutation action of \(S_{n}\) on \([n]\), decomposes as
\[M^{(n-1,1)} \cong S^{(n-1,1)} \oplus S^{(n)},\]
and therefore
\(\xi_{(n-1,1)} = \chi_{(n-1,1)}+1\).

\remove{
Clearly, if $K_{\lambda, \mu} \neq 0$ then $\lambda \domgeq \mu$ (if \(t\) is a semistandard generalized \(\lambda\)-tableau of content \(\mu\), then all \(\mu_{j}\) \(j\)'s must appear in the first \(j\) rows of \(t\), so \(\sum_{j=0}^{i} \mu_{j} \leq \sum_{j=0}^{i} \lambda_{j}\ \forall i\)). Also, $K_{\lambda, \lambda} = 1\ \forall \lambda \vdash n$, since the generalized \(\lambda\)-tableau with \(\lambda_{i}\) i's in the \(i\)th row is the unique semistandard generalized \(\lambda\)-tableau of content \(\lambda\). Order the partitions of \(n\) under the lexicographic order on partitions, greatest first. Since the lexicographic order extends the dominance order, \(K\) is then upper-triangular, and it has 1's all along the diagonal. Hence, \(K\) is invertible, so any permutation character \(\xi_{\beta}\) can be expressed as a linear combination of irreducible characters \(\chi_{\alpha}\). For any finite group \(G\), the irreducible characters of \(G\) form a basis for the !
 space of complex-valued class functions on \(G\) (functions constant on each conjugacy-class of \(G\)). The characters of \(S_{n}\) are real-valued, so the irreducible characters form a basis for the space of real-valued class functions. Hence, by the above argument, so do the permutation characters. This will be useful to bear in mind in section 4.
}

The restriction of an irreducible representation of \(S_{n}\) to the subgroup \(\{\sigma \in S_{n}: \sigma(i)=i \ \forall i > n-k\} = S_{n-k}\) can be decomposed into irreducible representations of \(S_{n-k}\) as follows:

\begin{theorem} \label{thm:branching-rule}
  (The branching rule.) Let \(\alpha\) be a partition of \(n-k\), and \(\lambda\) a partition of \(n\). We write $\alpha \subset^k \lambda$ if the Young diagram of $\alpha$ can be produced from that of \(\lambda\) by sequentially removing \(k\) corners (so that after removing the \(i\)th corner, we have the Young diagram of a partition of \(n-i\).) Let \(a_{\alpha,\lambda}\) be the number of ways of doing this; then we have
  \[[\lambda] \downarrow S_{n-k} = \sum_{\alpha \vdash n-k:\ \alpha \subset^{k} \lambda} a_{\alpha,\lambda}[\alpha],\]
  and therefore
  \[\chi_{\lambda} \downarrow S_{n-k} = \sum_{\alpha \vdash n-k:\ \alpha \subset^{k} \lambda} a_{\alpha, \lambda}\chi_{\alpha}.\]
  \end{theorem}

\begin{definition}
Let \(\lambda = (\lambda_{1},\ldots,\lambda_{k})\) be a partition of \(n\); if its Young diagram has columns of lengths \(\lambda_{1}' \geq \lambda_{2}' \geq \ldots \geq \lambda_{l}' \geq 1\), then the partition \(\lambda^{t} = (\lambda_{1}'\ldots,\lambda_{l}')\) is called the {\em transpose} of \(\lambda\), as its Young diagram is the transpose of that of \(\lambda\).
\end{definition}

\begin{theorem}
\label{thm:transpose-sign}
Let \(\lambda\) be a partition of \([n]\); then \([\lambda] \otimes [1^{n}] = [\lambda^{t}]\). Hence, \(\chi_{\lambda^{t}} = \chi_{\lambda} \cdot \textrm{sgn}\), and \(\dim[\lambda] = \dim[\lambda^{t}]\).
\end{theorem}

\begin{definition}
  The {\em hook} of a cell $(i,j)$ in the Young diagram of a partition $\mu$
  is $H_{i,j} = \{ (i,j') : j' \geq j \} \cup \{ (i',j) : i' \geq i \}$.
  The {\em hook length} of $(i,j)$ is $h_{i,j} = |H_{i,j}|$.
\end{definition}

\begin{theorem}[Frame, Robinson, Thrall \cite{frt}]
\label{hook-formula}
 If \(\lambda\) is a partition of \(n\) with hook lengths $(h_{i,j})$, then
  \begin{equation}
  \label{eq:hook}
    \dim [\lambda] = \frac{n!}{\prod_{i,j} h_{i,j}}.
  \end{equation}
\end{theorem}

\subsection{Lemmas regarding representations and characters}
In this subsection we state and prove several lemmas regarding
representations of $S_n$ and their characters; these will be
instrumental in proving our main theorem.

\subsubsection{Dimensions of irreducible representations}
\begin{lemma} \label{thm:dimension-for-short-2}
  Let \(k \in \mathbb{N}\). Then there exists \(E_k >0\) depending on \(k\) alone such that for any irreducible representation \([\lambda]\) of \(S_n\) with all rows and columns of length greater than $n-k$, $\dim[\lambda] \geq E_{k}n^{k+1}$.
\end{lemma}
To prove this lemma, we need two simple claims dealing with
irreducible representations with a relatively long row or column, and a separate
result dealing with the rest.

\begin{claim} \label{thm:dimension-for-long}
  Let $[\lambda]$ be an irreducible representation whose first row or column is of length
  $n-t$. Then
  \begin{equation} \label{eqn:dimension-for-long}
    \dim [\lambda] \geq \binom{n}{t} e^{-t}.
  \end{equation}
\end{claim}
\begin{proof}
  Note that if \(\lambda\) has first column of length \(n-t\), then \(\lambda^{t}\) has first row of length \(n-t\). Since \(\dim[\lambda] = \dim[\lambda^{t}]\), we may assume that \(\lambda\) has first row of length \(n-t\).

By the hook formula (\ref{eq:hook}), it suffices to prove that
\[\prod_{i,j} h_{i,j} \leq t!(n-t)! e^{t}.\]
Delete the first row \(R_1\) of the Young diagram of \(\lambda\); the resulting Young diagram \(D\) corresponds to a partition of \(t\), and therefore a representation of \(S_t\), which has dimension
\[\frac{t!}{\prod_{(i,j) \in D} h_{i,j}} \geq 1.\]
Hence,
\[\prod_{(i,j) \in D} h_{i,j} \leq t!.\]
We now bound the product of the hook lengths of the cells in the first row; this is of the form
\[\prod_{(i,j) \in R_1} h_{i,j} = \prod_{j=1}^{n-t}(j+c_j),\]
where \(\sum_{j=1}^{n-t}c_j = t\). Using the AM/GM inequality, we obtain:
\begin{eqnarray*}\prod_{j=1}^{n-t} \frac{j+c_j}{j}
    & = &
    \prod_{j=1}^{n-t} \left( 1 + \frac{c_j}{j} \right) \\
     & \leq &
    \left( \sum_{j=1}^{n-t} \frac{1 + \frac{c_j}{j}}{n-t} \right)^{n-t} \\
    & \leq &
    \left( \frac{n-t + \sum_{j=1}^{n-t} c_j}{n-t} \right)^{n-t} \\
    & = &
    \left( 1 + \frac{t}{n-t} \right)^{n-t}\\
 & < & e^t.
  \end{eqnarray*}
Hence,
\[\prod_{i,j} h_{i,j} \leq t!(n-t)! e^{t},\]
as desired.
\end{proof}

Note that, if $t$ is suffiently small depending on \(n\), \({n \choose t}e^{-t}\) is an increasing function of \(t\):
\begin{claim} \label{thm:dimension-increases}
  Let $L(n,t) = {n \choose t} e^{-t}$. Then $L(n,t) \leq L(n,t+1)$
  for all $t \leq (n-e)/(e+1)$.
\end{claim}
\begin{proof}
Observe that
  \[
    \frac{L(n,t)}{L(n,t+1)}=
    \frac{e(t+1)}{n-t}.\]
  Solving for when this expression is at most 1 proves the claim.
\end{proof}

For the representations not covered by Claim
\ref{thm:dimension-for-long}, we use the following.
\begin{theorem}[\cite{mish96}]
\label{thm:dimension-for-short}
   If $\alpha, \epsilon > 0$, then there exists \(N(\alpha,\epsilon) \in \mathbb{N}\) such that for all $n >
  N(\alpha,\epsilon)$, any irreducible representation \([\lambda]\) of $S_n$
  which has all rows and columns of length at most $n / \alpha$ has
\[\dim[\lambda] \geq (\alpha-\epsilon)^n.\]
\end{theorem}

The proof of Lemma \ref{thm:dimension-for-short-2} is now immediate.

\begin{proof}[Proof of Lemma \ref{thm:dimension-for-short-2}:]
  If the partition \(\lambda\) contains a row or column of length between 
  $n-(n-e)/(e+1)$ and \(n-k-1\), then by Claims \ref{thm:dimension-for-long} and \ref{thm:dimension-increases},
\[\dim[\lambda] \geq e^{-(k+1)}{n \choose k+1} \geq E_k n^{k+1},\]
provided we choose \(E_k >0\) sufficiently small. Otherwise, the conditions of Theorem \ref{thm:dimension-for-short}
  hold with $\alpha = (e+1)/e - \epsilon'$ for some small
  $\epsilon' >0$, and therefore, by Theorem \ref{thm:dimension-for-short},
\[\dim[\lambda] \geq n^{k+1},\]
provided \(n\) is sufficiently large depending on \(k\), completing the proof.
\end{proof}

\subsubsection{Character tables and their minors}
We will be working with certain minors of the character table of \(S_n\). The following lemmas imply that certain related matrices are upper-triangular with 1's all along the diagonal.

\begin{lemma}
 \label{lemma:kostka}
If \(\lambda\),\(\mu\) are partitions of \(n\), let \(K_{\lambda,\mu}\) denote the Kostka number, the number of semistandard \(\lambda\)-tableaux of content \(\mu\). If \(K_{\lambda,\mu} \geq 1\), then \(\lambda \geq \mu\). Moreover, \(K_{\lambda,\lambda}=1\) for all \(\lambda\).
\end{lemma}
\begin{proof}
Observe that if there exists a semistandard generalized \(\lambda\)-tableau of content \(\mu\), then all \(\mu_{i}\) \(i\)'s must appear in the first \(i\) rows, so \(\sum_{j=0}^{i} \mu_{j} \leq \sum_{j=0}^{i} \lambda_{j}\) for each \(i\). Hence, \(\lambda \domgeq \mu\). Since the lexicographic order extends the dominance order, it follows that \(\lambda \geq \mu\). Observe that the generalized \(\lambda\)-tableau with \(\lambda_{i}\) \(i\)'s in the \(i\)th row is the unique semistandard generalized \(\lambda\)-tableau of content \(\lambda\), so \(K_{\lambda,\lambda}=1\) for every partition \(\lambda\).
\end{proof}

\begin{lemma} \label{thm:dominance-lemma}
  Let $\lambda$ be a partition of $n$, and let $\xi_{\lambda}$ be the character of
  the permutation module $M^{\lambda}$.  Let $\sigma \in S_n$. If
  $\xi_{\lambda}(\sigma) \neq 0$, then $\textrm{cycle-type}(\sigma) \domleq \lambda$. Moreover, if \(\textrm{cycle-type}(\sigma) = \lambda\), then \(\xi_{\lambda}(\sigma) = 1\).
\end{lemma}

\begin{proof}
  The set of $\lambda$-tabloids is a basis for the permutation module $M^{\lambda}$. Thus, $\xi_{\lambda}(\sigma)$, which is the trace of the corresponding representation on the permutation \(\sigma\), is simply the number of \(\lambda\)-tabloids fixed by \(\sigma\). If $\xi_{\lambda}(\sigma) \neq 0$, then $\sigma$ fixes some \(\lambda\)-tabloid \([t]\). Hence, every row of length $l$
  in \([t]\) is a union of the sets of numbers in a collection of disjoint cycles of total
  length $l$ in $\sigma$. Thus, the cycle-type of \(\sigma\) is a refinement of \(\lambda\), and therefore $\textrm{cycle-type}(\sigma) \domleq \lambda$, as required. If \(\sigma\) has cycle-type \(\lambda\), then it fixes just one \(\lambda\)-tabloid, the one whose rows correspond to the cycles of \(\sigma\), so \(\xi_{\lambda}(\sigma) = 1\). 
\end{proof}

\begin{theorem} \label{thm:minor-is-invertible}
  Let $C$ be the character table of $S_n$, with rows and columns
  indexed by partitions / conjugacy classes in decreasing lexicographic
  order (so $C_{\lambda, \mu} = \chi_{\lambda}(\sigma_{\mu})$ where
  $\sigma_{\mu}$ is a permutation with cycle-type $\mu$, and the
  top-left corner of $C$ is $\chi_{(n)}(\sigma_{(n)})$.) Then the
  contiguous square minor \(\tilde{C}\) of \(C\) with rows and columns $\psi : \psi >
  (n-k,1^k)$ is invertible and does not depend on $n$, provided \(n > 2k\).
\end{theorem}

\begin{proof}
  Let \(K\) be the Kostka matrix, and let \(D\) be the matrix of permutation characters,
\begin{equation}
    D_{\lambda, \mu} = \xi_{\lambda}(\sigma_{\mu}),
  \end{equation}
  where $\sigma_{\mu}$ denotes a permutation with cycle-type $\mu$. Let \(\tilde{K}\) and \(\tilde{D}\) denote the top-left minor of \(K\) and \(D\) respectively (i.e. the minor with rows and columns $\psi : \psi >
  (n-k,1^k)$).

  Recall that by Young's rule (Theorem
  \ref{thm:young-rule}), we have
  \begin{equation}
    M^{\mu} \cong \oplus_{\lambda} K_{\lambda, \mu} S^{\lambda},
  \end{equation}
  and therefore
  \begin{equation} \label{eqn:decomposition-of-m-mu}
    \xi_{\mu} = \sum_{\lambda} K_{\lambda, \mu} \chi_{\lambda}.
  \end{equation}
  Hence,
  \begin{equation}
\label{eq:ident}
    (K^t C)_{\lambda,\mu} = \sum_{\tau} K_{\tau, \lambda} C_{\tau, \mu} =
    \sum_{\tau} K_{\tau, \lambda} \chi_{\tau}(\sigma_{\mu}) = \xi_{\lambda}(\sigma_\mu) =
    D_{\lambda,\mu},
  \end{equation}
  and therefore
  \begin{equation}
\label{eq:prod}
    K^t C = D.
  \end{equation}

Since the rows and columns of \(K\) are sorted in decreasing lexicographic order, \(K\) is upper-triangular with 1's all along the diagonal, by Lemma \ref{lemma:kostka}. Therefore, \(K^{t}\) is lower-triangular with 1's all along the diagonal.

Since \(K^t\) is lower-triangular, in addition to (\ref{eq:prod}), we also have
\begin{equation}
\tilde{K}^t \tilde{C} = \tilde{D}.
\end{equation}
Since \(\tilde{K}^t\) is lower-triangular with 1's all along the diagonal, it is invertible, and therefore
\[\tilde{C} = (\tilde{K}^t)^{-1}\tilde{D}.\]
By Lemma \ref{thm:dominance-lemma}, \(\tilde{D}\) is upper-triangular with 1's all along the diagonal, and is therefore invertible. It follows that \(\tilde{C}\) is invertible also.

We will now show that \(\tilde{K}\) and \(\tilde{D}\) are independent of \(n\), provided \(n > 2k\); this will prove that \(\tilde{C}\) is also independent of \(n\).

 Let $\lambda >
  (n-k,1^k)$ be a partition. Then $\lambda_1 \geq n-k$.  Write
  $\lambda' = (\lambda_1 - (n-k), \lambda_2, \ldots)$. (Note that this may not be a {\em bona fide} partition, as it may not be in non-increasing order.) Now the
  mapping $\lambda \mapsto \lambda'$ has the same image over $\{
  \lambda : \lambda > (n-k,1^k) \}$ for all $n \geq 2k$:
  namely, `partitions' of $k$ where the first row is not necessarily the
  longest.

We first consider \(K\). Recall once again that $K_{\lambda,\mu}$ is the number of semistandard \(\lambda\)-tableaux of content $\mu$. Let $t$ be a semistandard
  \(\lambda\)-tableau of content $\mu$; we now count the number of choices for \(t\). Since the numbers in a semistandard
  tableau are strictly increasing down each column and non-increasing along each row, and \(\mu_1 \geq n-k\), we must always place 1's in the first \(n-k\) cells of the first row of \(t\). We must now fill the rest of the cells with content \(\mu'\). Provided \(n \geq 2k\), \(\mu'\) is independent of \(n\), and the remaining cells in the first row have no cells below them, so the number of ways of doing this is independent of \(n\). Hence, the entire minor \(\tilde{K}\) is independent of $n$.

  Now consider $D$. Recall that \(D_{\lambda,\mu} = \xi_{\lambda}(\sigma_{\mu})\) is simply the number of \(\lambda\)-tabloids fixed by $\sigma_{\mu}$. To count these, first note that the numbers in the long cycle
  of $\sigma$ (which has length at least \(n-k\)) must all be in the first row of the \(\lambda\)-tabloid (otherwise
  the long cycle of $\sigma$ must intersect two or more rows, as \(n-k > k\).)
  This leaves us with a $(\lambda_1 - \mu_1, \lambda_2,
  \ldots, \lambda_r)$-`tableau', which we need to fill with the remaining $n -
  \mu_1$ elements in such a way that $\sigma$ fixes it.  It is
  easy to see that, again, the number of ways of doing this is independent of \(n\).
\end{proof}

In particular, if \(n \geq 2k\), the number of partitions \(\lambda\) of \(n\) such that \(\lambda \geq (n-k,1^k)\) is independent of \(n\); we denote it by \(q_k\). Note that
\[q_k = \sum_{t=0}^{k}p_t,\]
where \(p_t\) denotes the number of partitions of \(t\).

We need a slightly more general result, which allows us to split some of the partitions.

\begin{definition}\label{split}
  Assume that \(n > 3k+1\), and let $\mu = (\mu_1, \ldots, \mu_r)$ be a partition of $n$ with $\mu_1 \geq n-k$. We define
$$\Split(\mu) = (\mu_1 - k-1, k+1, \mu_2, \ldots, \mu_r).$$
\end{definition}

It is easy to see that $\Split(\mu)$ is indeed a partition (i.e. it is in
descending order). Further, exactly one of $\mu$ and $\Split(\mu)$ is
even.

\begin{theorem} \label{thm:minor-is-invertible-2}
  Let $C$ be as above, and let \(n > 3k+1\). Let $\phi_1, \ldots, \phi_{q_k-1}$ be the
  partitions $>(n-k,1^k)$. Let $\mu_1, \ldots,
  \mu_{q_k-1}$ be partitions such that, for each $j$, either
  $\mu_j = \phi_j$ or else $\mu_j = \Split(\phi_j)$. Then the square minor \(\breve{C}\) of $C$ with \(i\)th row
  $\phi_i$ and \(j\)th column $\mu_j$ is independent of the choices of the \(\mu_j\)'s, so is always equal to the top-left minor \(\tilde{C}\).
\end{theorem}

\begin{proof}
  It is easy to see that if \(\lambda,\mu > (n-k,1^{k})\), \(\sigma\) has cycle-type \(\mu\) and \(\sigma'\) has cycle-type \(\Split(\mu)\), then in fact, \(\xi_{\lambda}(\sigma) = \xi_{\lambda}(\sigma')\). (All rows of a \(\lambda\)-tabloid below the first have length at most \(k\). Hence, if a permutation \(\sigma'\) with cycle-type \(\Split(\mu)\) fixes a \(\lambda\)-tabloid \([t]\), the numbers in \((\mu_1-k-1)\)-cycle and the \((k+1)\)-cycle must all lie in the first row of \([t]\). It follows that a permutation \(\sigma\) produced by merging these two cycles of \(\sigma'\) fixes exactly the same \(\lambda\)-tabloids as \(\sigma'\) does.) Since the Kostka matrix \(K\) is upper-triangular with 1's all down the diagonal, \(\{\xi_{\lambda}:\ \lambda > (n-k,1^k)\}\) and \(\{\chi_{\lambda}:\ \lambda > (n-k,1^k)\}\) are bases for the same linear space, and therefore
\[\chi_{\lambda}(\sigma) = \chi_{\lambda}(\sigma')\]
for each \(\lambda > (n-k,1^k)\). Thus,
\[C_{\lambda,\Split(\mu)} = C_{\lambda,\mu}\quad \forall \lambda,\mu > (n-k,1^k),\]
as required.
\end{proof}

\subsubsection{Functions with Fourier transform concentrated
  on the `fat' irreducible representations}

One of the recurring themes in applications of discrete Fourier
analysis to combinatorics is proving that certain functions depend on
few coordinates, by showing that their Fourier transform is
concentrated on the `low frequencies', the characters indexed by small
sets. Examples of this can be found for example in \cite{Bourgain}, \cite{F98} and \cite{FKN}. In this paper, we need a non-Abelian analogue. We show that
functions on $S_n$ whose Fourier transform is supported on
irreducible representations which are large with respect to the
lexicographic order, are spanned by the cosets of the pointwise stabilizers of small
sets.

\begin{definition}
  Let $V_k$ be the linear space of functions whose Fourier transform
  is supported only on representations $\geq (n-k,1^k)$.
\end{definition}
We are now ready to prove
\begin{thm:repeat}
  $V_k$ is the span of the $k$-cosets.
\end{thm:repeat}
\begin{proof}
  First, we show that the characteristic function of any $k$-coset is
  indeed in $V_k$. Let $T = T_{a_1 \mapsto b_1, \ldots, a_k
  \mapsto b_k}$. It is easy to check that if \(f:S_n \to \mathbb{R}\), \(\tau,\pi \in S_n\), and
\begin{eqnarray*}
g:S_n & \to & \mathbb{R}\\
\sigma & \mapsto & f(\pi \sigma \tau),
\end{eqnarray*}
i.e. the function \(g\) is a `double-translate' of \(f\), then \(\hat{f}(\rho) = 0\Rightarrow \hat{g}(\rho)=0\). Hence, if \(f \in V_k\), so is \(g\). (This is saying that \(V_k\) is a `two-sided ideal' of the group algebra \(\mathbb{R}[S_n]\).)

Hence, by double-translation, without loss of generality, we may assume that \(a_i = b_i = i\) for each \(i \in [k]\), i.e. \(T = T_{1 \mapsto 1,\ldots,k \mapsto k}\), so \(T \cong S_{n-k}\). We use
  the branching rule (Theorem \ref{thm:branching-rule}.) If $\rho <
  (n-k,1^k)$, then $\rho$ has at least $k+1$ cells outside the first
  row. In that case, every $\mu \subset^k \rho$ is a nontrivial irreducible
  representation of $S_{n-k}$. We claim that this implies
\[\sum_{\sigma \in S_{n-k}} \mu(\sigma)=0.\]
To see this, observe that the linear map \(\sum_{\sigma \in S_{n-k}} \mu(\sigma)\) commutes with \(\mu\), and therefore by Schur's Lemma, it must be a scalar multiple of the identity map. Its trace is zero, since \(\chi_{\mu}\) is orthogonal to the trivial character, so it must be the zero map. Hence, $\widehat{1_T}(\rho) = 0$ for each \(\rho < (n-k,1^k)\).

  (Note also that if $\rho = (n-k,1^k)$, then the only $\mu
  \subset^k \rho$ with nonzero sum on $T$ is that corresponding to the
  partition $(n-k)$, i.e. the trivial representation, which has sum
  $(n-k)!$ on $T$. Therefore, \(\widehat{1_{T}}((n-k,1^k)) \neq 0\). This will be useful in the proof of Theorem \ref{thm:linear-exist}.)

  In the other direction, let $f \in V_k$. Using the Fourier inversion
  formula, we have
  \begin{equation}
    f(\sigma) =
    \sum_{\rho \geq (n-k,1^k)}
         \dim[\rho] \tr \left[ \Hat{f}(\rho) \rho(\sigma^{-1}) \right].
  \end{equation}
  Consider the permutation module $M^{(n-k,1^k)}$. By Young's Rule (Theorem \ref{thm:young-rule}), $M^{(n-k,1^k)}$ must contain at least one copy of every Schur module
  $S^{\rho}$ with $\rho \geq (n-k,1^k)$,
  and no others. We can therefore rewrite the previous formula as:
  \begin{equation} \label{eqn:what-we-need}
    f(\sigma) = \tr \left[ A \psi(\sigma^{-1}) \right],
  \end{equation}
  where $\psi$ is a matrix representation corresponding to
  $M^{(n-k,1^k)}$, and $A$ is a block-diagonal matrix whose blocks correspond to
  the irreducible modules $S^{\rho}$, i.e. it has $K_{\rho,(n-k,1^k)}$ blocks corresponding to $\rho$,
  all equal to
  $$
  \frac{\dim[\rho]}{K_{\rho,(n-k,1^k)}}\Hat{f}(\rho),
  $$
  where the Kostka number $K_{\rho, (n-k,1^k)}$ is the multiplicity of $S^{\rho}$ in $M^{(n-k,1^k)}$.

  Next, recall that the permutation module $M^{(n-k,1^k)}$ corresponds to the permutation action of $S_n$
  on $(n-k,1^k)$-tabloids, which can be identified
  with ordered $k$-tuples of distinct numbers between 1 and \(n\). Since trace is conjugation-invariant, we can
  perform a change of basis, replacing $A$ and $\psi$ by similar
  matrices $B$ and $\phi$, so that
   \begin{equation} \label{eqn:what-we-need-2}
    f(\sigma) = \tr \left[B \phi(\sigma^{-1})\right],
  \end{equation}
  where $\phi_{\alpha,\beta}(\tau)$ is 1 if $\tau$ takes the ordered $k$-tuple
  $\alpha$ to the ordered $k$-tuple $\beta$, and 0 otherwise.  But then
  equation (\ref{eqn:what-we-need-2}) means precisely that
$$f =
  \sum_{\alpha,\beta} B_{\alpha,\beta}T_{\alpha \mapsto \beta},$$
completing the proof.
\end{proof}

\section{Proof of main theorems} \label{sec:main-proofs}
\subsection{Proof strategy}
In this section we finally proceed to eat the pudding. In view of
Theorem \ref{thm:cosets-span}, we would like to prove that whenever $I$ is a maximum-sized
$k$-intersecting family, the Fourier transform of its characteristic
function is supported on the irreducible representations corresponding
to partitions $\rho \geq (n-k,1^k)$, i.e. those whose Young diagram has first row of length at least \(n-k\).  Let us call
these representations the `fat' representations, and their
transposes (those whose Young diagram has first column of height at least \(n-k\)) the `tall' representations. The rest will be called `medium' representations.

Given a $k$-intersecting family $I \subset S_n$, we will first consider it as an independent
set in the Cayley graph \(\Gamma_{k}\) on $S_n$ generated by $\fpf_k$, the
set of permutations with less than $k$ fixed points. Since \(\fpf_k\) is a union of conjugacy classes, by Theorems \ref{thm:cayley-spectrum} and \ref{thm:spechtmodules}, the eigenvalues of \(\Gamma_{k}\) are given by
\begin{equation}\lambda^{(k)}_{\rho} = \frac{1}{\dim[\rho]} \sum_{\sigma \in \fpf_k} \chi_{\rho}(\sigma)\quad (\rho \vdash n).
\end{equation}

Fixed-point-free permutations are also called {\em derangements}, and \(\Gamma_1\) is also called the {\em derangement graph}. Let \(d_{n} = |\textrm{FPF}_{1}(n)|\); it is well-known (and easy to see, using the inclusion-exclusion formula), that
\[d_{n}= \sum_{i=0}^{n} (-1)^{i}{n \choose i}(n-i)! = \sum_{i=0}^{n} (-1)^{i}\frac{n!}{i!} =  (1/e+o(1))n!.\]
For \(n \geq 5\), the eigenvalues of \(\Gamma_1\) satisfy:
\begin{eqnarray*}
\lambda^{(1)}_{(n)}& = &d_{n}\\
\lambda^{(1)}_{(n-1,1)} &=& -d_{n}/(n-1)\\
|\lambda_{\rho}^{(1)}| &<& cd_{n}/n^{2} < d_{n}/(n-1)\quad \textrm{for all other }\rho \vdash n
\end{eqnarray*}
(where \(c\) is an absolute constant). Hence, the matrix \(\Gamma_{1}\) has \[\lambda_{\textrm{\scriptsize{min}}}/\lambda_{1} = -1/(n-1).\]
As observed in \cite{godsil-2008} and \cite{renteln}, this implies via Hoffman's bound (Theorem \ref{thm:hoffman}) that any 1-intersecting family \(I \subset S_{n}\) satisfies \(|I| \leq (n-1)!\). If equality holds, then \(1_{I} \in V_{1}\), from which we may conclude (e.g. from the \(k=1\) case of Theorems \ref{thm:cosets-span} and \ref{thm:span-is-disjoint-union}) that \(I\) is a 1-coset. Also, we may conclude via Theorem \ref{thm:Xhoffman} that any 1-cross-intersecting pair of families \(I,J \subset S_{n}\) safisfy \(|I||J| \leq ((n-1)!)^{2}\). If equality holds, then \(1_{I}, 1_{J} \in V_{1}\), which enables us to conclude that \(I=J\) is a 1-coset of \(S_{n}\). This proves Leader's conjecture on 1-cross-intersecting families in \(S_{n}\) for \(n \geq 5\) (it can be verified directly for \(n=4\)).

However, for \(k\) fixed and \(n\) large, calculating the least eigenvalue of \(\Gamma_{k}\) and applying Hoffman's bound only gives an upper bound of \(\Theta((n-1)!)\) for the size of a \(k\)-intersecting subset of \(S_n\). Indeed,
\begin{eqnarray*}
\lambda^{(k)}_{(n-1,1)} & = & \frac{1}{\dim[n-1,1]}\sum_{\sigma \in \textrm{FPF}_{k}} \chi_{(n-1,1)}(\sigma)\\
& = & \frac{1}{n-1} \sum_{\sigma \in \textrm{FPF}_k} (\xi_{(n-1,1)}(\sigma)-1)\\
& = & \frac{1}{n-1}\sum_{i=0}^{k-1} {n \choose i}d_{n-i}(i-1)\\
& = & -\frac{1}{n-1}(1/e+o(1))n!\left(1-\sum_{i=1}^{k-2} \frac{i}{(i+1)!}\right).
\end{eqnarray*}
Note that
\[\sum_{i=1}^{\infty} \frac{i}{(i+1)!} = \frac{d}{dx} \frac{e^{x}-1}{x} \bigg|_{x=1} = 1,\]
so for any \(k \in \mathbb{N}\),
\[1-\sum_{i=1}^{k-2} \frac{i}{(i+1)!} >0,\]
and therefore \(\lambda^{(k)}_{(n-1,1)} = -\Theta_{k}((n-1)!)\). It turns out that \(|\lambda^{(k)}_{\rho}| \leq a_{k} n! / n^2\) for all partitions \(\rho \neq (n),(n-1,1)\), where \(a_{k} >0\) depends on \(k\) alone. Hence, if \(k\) is fixed and \(n\) is large, \(\Gamma_k\) has
\[\lambda_{\min} = \lambda^{(k)}_{(n-1,1)} = -\Theta_{k}((n-1)!).\]

Hence, applying Hoffman's bound only gives
\[|I| \leq \Theta_{k}((n-1)!),\]
for a \(k\)-intersecting \(I \subset S_n\).

Instead, we will construct a linear combination \(Y\) of subgraphs of \(\Gamma_{k}\) (each a Cayley graph generated by a conjugacy-class within \(\fpf_k\)) which has the correct eigenvalues for us to prove Theorems \ref{main-theorem} and \ref{thm:span-is-maximal} from Theorem \ref{thm:hoffman-linear}. By Theorem \ref{thm:cayley-spectrum}, if \(X_1,\ldots,X_t\) are conjugacy-classes within \(\fpf_k\), and \(\beta_1,\ldots,\beta_t \in \mathbb{R}\), then the eigenvalues of the linear combination
\[Y = \sum_{j=1}^{t} d_j \Cay(S_n,X_j)\]
are
\begin{equation}
\label{eq:linearcomb}
\lambda_{\rho} = \frac{1}{\dim[\rho]} \sum_{j=1}^{t} d_j \sum_{\sigma \in X_j} \chi_{\rho}(\sigma) = \frac{1}{\dim[\rho]} \sum_{j=1}^{t} d_j |X_j| \chi_{\rho}(\tau_j),\quad (\rho \vdash n)
\end{equation}
where \(\tau_j\) denotes any permutation in \(X_j\). So the eigenvalues of \(Y\) still correspond to partitions of \(n\), and are therefore relatively easy to handle.

The value of $\lambda_{\min}/\lambda_1$ required in Theorem \ref{thm:hoffman-linear} to produce the upper bound in Theorem \ref{main-theorem} is as follows:

\begin{cor} \label{thm:hoffman-calc}
  Define
$$\omega = \omega_{n,k} = -\frac{(n-k)!}{n!-(n-k)!} = -\frac{1}{n(n-1)\ldots(n-k+1)-1}.$$
  Assume the conditions of Theorem \ref{thm:hoffman-linear}. If $\lambda_{min} / \lambda_{1} = \omega$, and $I$ is an independent set in $G$, then $|I| \leq (n-k)!$.
\end{cor}
\begin{proof}
Immediate from Theorem \ref{thm:hoffman-linear}.
\end{proof}

Since rescaling the linear combination of graphs makes no difference to the above, our aim will be to construct a linear combination \(Y\) with \(\lambda_1=1\) and \(\lambda_{\min} = \omega\). Since equality has to hold in Theorem \ref{thm:hoffman-linear} when the independent set is a \(k\)-coset, we know that for any \(k\)-coset \(T\), we must have
\[1_{T} \in \Span \left( \{ v_1 \} \cup \{ v_i : \lambda_i = \omega \} \right).\]
By Theorem \ref{thm:cosets-span}, it follows that we must have \(\lambda_{\rho} = \omega\) for each fat partition \(\rho \neq (n)\).

We will in fact construct two linear combinations, $\yeven$ and $\yodd$, from Cayley graphs generated by respectively even/odd conjugacy-classes within \(\fpf_k\). We design these linear combinations so
that \(\lambda_{(n)} = 1\) and \(\lambda_{\rho} = \omega\) for all fat \(\rho \neq (n)\). Recalling for any partition \(\rho\), we have $\chi_{\rho^{t}}= \chi_{\rho} \cdot
sgn$, where $sgn$ is the sign character, we see from (\ref{eq:linearcomb}) that for any partition \(\rho\), we have:
\begin{itemize}
 \item \(\lambda_{\rho^t} = \lambda_{\rho}\) in \(\yeven\), whereas
\item \(\lambda_{\rho^t} = -\lambda_{\rho}\) in \(\yodd\).
\end{itemize}
Hence, simply taking
\[Y = \tfrac{1}{2}(\yeven+\yodd)\]
will ensure that the eigenvalues of \(Y\) corresponding to the
tall representations are all zero.

Moreover, we will show that in both $\yeven$ and $\yodd$ (and consequently in \(Y\)), the eigenvalues corresponding to the medium
representations all have absolute value \(|\lambda| \leq c_{k} |\omega|/\sqrt{n} = o(|\omega|)\), where \(c_k >0\) depends on \(k\) alone. So provided \(n\) is sufficiently large, \(\omega\) is both the minimal eigenvalue of \(Y\) and the second-largest in absolute value, and is attained only on the non-trivial fat irreducible representations. The situation is summarized in Table \ref{table:eigenvalues}; note that the \(o(|\omega|)\) function is always the same function.

\begin{table}
\label{table:eigenvalues}
\center{
\begin{tabular}{|c|c|c|c|c|c|}
  \hline
            & $(n)$ & fat, $\neq (n)$      & $(1^n)$ & tall, $\neq (1^n)$     & medium        \\ \hline
  $\yeven$  & 1       & $\omega$ &  1   & $\omega$  & $o(|\omega|)$   \\ \hline
  $\yodd$   & 1       & $\omega$ & $-1$   & $-\omega$ & $o(|\omega|)$   \\ \hline
  $Y$       & 1       & $\omega$ &  0   & 0         & $o(|\omega|)$   \\ \hline
\end{tabular}
}
\caption{Eigenvalues}
\end{table}

Applying Theorem \ref{thm:hoffman-linear} to $Y$ will not only prove that $|I| \le (n-k)!$, but also that if equality holds, then the Fourier transform of the characteristic function of \(I\) is totally supported
on the fat representations, yielding the proof of our main
theorem (pending the results of section
\ref{sec:boolean-functions}). Since $\omega$ is also the eigenvalue of second-largest absolute value, we will also be able to deduce Theorem \ref{thm:X-main}, concerning $k$-cross-intersecting families.

In order to carry out our plan, we will identify appropriate conjugacy-classes \(X_j\) to use as the generating sets for our Cayley graphs, and appropriate coefficients \(d_j\). The set of linear equations the \(d_j\)'s must satisfy will always correspond to a specific square minor of the character table of $S_n$, which turns out to be non-singular and independent of \(n\) (for \(n\) sufficiently large). The latter statement is precisely the content of Theorems \ref{thm:minor-is-invertible} and \ref{thm:minor-is-invertible-2}.

As for the medium representations, we will show that their eigenvalues
have small absolute value using the lower bound we proved on their
dimensions in Lemma \ref{thm:dimension-for-short-2}.

\subsection{Calculations of eigenvalues of Cayley graphs}
The Cayley graphs we will use to construct our linear combination will be generated by conjugacy-classes of permutations with specific cycle-types. These cycle-types will fall into two categories: those
corresponding to partitions $> (n-k,1^k)$,
and those corresponding to partitions $\pi$, where
$$(n-k-1,k+1) \geq \pi > (n-2k-1, k+1, 1^k).$$

Note that the second category is exactly
the split of the first category (see definition \ref{split}). We need to use the split partitions in order to ensure that our conjugacy-classes have the correct sign.

We will need the following bounds on the sizes of the above conjugacy-classes:

\begin{lemma} \label{lemma:cycle-measure}
  Let $X$ be a conjugacy class of $S_n$ with a cycle of
  length $n-t$, where $t < n/2$. Then
\[\frac{n!}{(n-t)t^t}\leq |X| \leq 2(n-1)!.\]
\end{lemma}
\begin{proof}
  Suppose the cycle-type of \(X\) is \((n-t,c_1,\ldots,c_l)\), where \(l \leq t\) and \(\sum_{i=1}^{l}c_i = t\). Define a mapping
\[F:\ S_n \to X\]
by taking a permutation \(\sigma \in S_n\), writing it in sequence-notation
\[\sigma(1),\sigma(2),\ldots,\sigma(n),\]
and then placing parentheses at the appropriate points, producing a permutation in \(X\) (written in disjoint cycle notation). This mapping is clearly surjective, and each element of \(X\) has the same number of preimages, \(N\) say. Note that the preimages of a permutation in \(X\) are obtained by rotating each cycle and then permuting the cycles of length \(j\) for each \(j\). Let \(a_j = |\{i:\ c_i = j\}|\) for each \(j\); then
\[N = (n-t)\prod_j a_j!\prod_{i=1}^{l}c_i.\]
Observe that
\[n/2 < n-t \leq N \leq (n-t)l!\prod_{i=1}^{l}c_i \leq (n-t)l!(t/l)^{l} \leq (n-t)t^l \leq (n-t)t^t,\]
using the AM/GM inequality. Hence,
\[\frac{n!}{(n-t)t^t}\leq |X| \leq 2(n-1)!,\]
as required.
\end{proof}

We now proceed to the key step of the proof.

\begin{theorem} \label{thm:linear-exist}
  If $\rho$ is a partition of $n$, let $X_{\rho}$ denote the conjugacy class of
  permutations with cycle-type $\rho$. Assume that \(n >3k+1\), and let \(q_{k}\) denote the number of partitions $\geq (n-k,1^k)$. Let $\phi_1, \ldots, \phi_{q_{k}-1}$ be the partitions of $n$
  which are $>(n-k,1^k)$ (in decreasing lexicographic order). Let $\mu_1, \ldots, \mu_{q_{k}-1}$ be
  such that for all $j$, either $\mu_j = \phi_j$ or else $\mu_j =
  \Split(\phi_j)$.  Let $G_j = \Cay(S_n,X_{\mu_{j}})$, and let $(\lambda_i^{(j)})_i$ be the eigenvalues of $G_j$,
  in decreasing lexicographic order of their representations (so
  that $\lambda_1^{(j)} = |X_{\mu_j}|$).

  Then there exist $d_1, \ldots, d_{q_{k}-1} \in \mathbb{R}$ such that:

  \begin{equation} \label{eqn:linear-combination-of-lambda}
  \sum_j d_j \lambda_i^{(j)} =
  \left\{
  \begin{array}{rl}
    1,      & \textrm{for } i = 1 \\
    \omega, & \textrm{for } 1 < i \leq q_k
  \end{array}
  \right.
  \end{equation}

  Moreover, there exists $B_k > 0$ depending only on $k$ such
  that
$$\max_j |d_j| \leq \frac{B_k}{(n-1)!}.$$
\end{theorem}
Theorem \ref{thm:linear-exist} will enable us to get the correct eigenvalues on the fat representations, but we will also need the following, which
ensures that the eigenvalues for the medium representations are smaller
in absolute value:

\begin{theorem}\label{thm:medreps}
  Under the assumptions and notation of Theorem \ref{thm:linear-exist},
  let $\rho$ be a medium partition, and let $\lambda_{\rho} =
  \sum_{j=1}^{q_k-1} d_j \lambda_{\rho}^{(j)}$. Then
$$|\lambda_{\rho}| \leq c_k|\omega|/\sqrt{n} = o(|\omega|),$$
  where \(c_k >0\) depends only on \(k\).
\end{theorem}

We defer the proof of this until later, and turn first to the proof of
Theorem \ref{thm:linear-exist}.

\begin{proof}[Proof of Theorem \ref{thm:linear-exist}:] There are three steps in the proof. Step 1 is showing the existence of coefficients $d_j$ such that
equation (\ref{eqn:linear-combination-of-lambda}) holds for $i < q_{k}$. Step 2 is showing that with these
coefficients, the equation also holds for $i=q_k$. Step 3 is obtaining the bound on $\max_{j}|d_j|$.

{\bf Step 1 ($i < q_{k}$):} Consider the system of linear equations $Md = b$, where $d = (d_1,
  \ldots, d_{q_{k}-1})$, $b = (1, \omega, \ldots, \omega)$, and
  $M_{ji} = \lambda_i^{(j)}$. By Theorem \ref{thm:cayley-spectrum},
  \begin{equation}
    \lambda_i^{(j)} = \frac{1}{\dim[\phi_i]} \sum_{\tau \in X_j} \chi_{\phi_i}(\tau) = \frac{|X_j|}{\dim[\phi_i]} \chi_{\phi_i}(\tau_j),
  \end{equation}
where \(\tau_j\) is a representative for the conjugacy-class \(X_j\) (recall that characters are class-functions). Equivalently,
  \[ M = N_1 \breve{C}^t N_2,\]
  where $\breve{C}$ denotes the minor of the character table of \(S_n\) with rows $\phi_1, \ldots, \phi_{q_{k}-1}$ and columns $\mu_1, \ldots, \mu_{q_{k}-1}$, and $N_1$ and $N_2$ are the
  diagonal row- and column- normalization matrices respectively; explicitly,
\begin{equation}
\label{eq:diagonal}
(N_1)_{i,j} = \frac{\delta_{i,j}}{\dim[\phi_i]},\quad (N_2)_{i,j} = \delta_{i,j}|X_j|.
\end{equation}
 Recall from Theorem \ref{thm:minor-is-invertible-2} that \(\breve{C}\) is always equal to the top-left minor \(\tilde{C}\), and that $\tilde{C}$ (and therefore \(\tilde{C}^t\)) is invertible. Therefore, there is a (unique) solution for $d$, so we can find
  appropriate values for $d_1, \ldots, d_{q_{k}-1}$.

{\bf Step 2 ($i = q_{k}$):}  Write
  $$\lambda= \sum_j d_j \lambda_{q_{k}}^{(j)}.$$
 We must show that $\lambda= \omega.$
  This will follow from analyzing the proof of the generalized Hoffman bound, Theorem \ref{thm:hoffman-linear}, when \(G=\Gamma_k\) and the independent set \(I\) is a \(k\)-coset.

Let $T$ be a $k$-coset of $S_n$, and let $f = 1_T$ be its characteristic function. Choose the orthonormal basis of eigenvectors \(\{v_i\}\) in Theorem \ref{thm:hoffman-linear} to be the normalization of the orthogonal basis formed by the entries of (matrix-equivalents of) the irreducible representations \([\rho]\) (see section \ref{subsec:cayleygraphs}). Observe that
\[\alpha = \mathbb{E}f = \frac{(n-k)!}{n!}.\]
By the argument in the proof of Theorem \ref{thm:cosets-span}, we see that
\begin{equation}
    \widehat{f}(\rho) \neq 0 \Rightarrow \rho \geq (n-k,1^k),
  \end{equation}
 and moreover, $\widehat{f}((n-k,1^k)) \neq 0$.

 Let $W(\rho) = Tr[\Hat{f}(\rho)(\Hat{f}(\rho))^t]$. This is simply the $L^2$-weight of the coefficients of eigenvectors \(v_i\) which correspond to entries of (the matrix-equivalent of) \([\rho]\), i.e. if we write as usual $f = \sum a_i v_i$,
 then $W(\rho) = \sum a_i^2$, where the sum is over all
 $i$ such that $v_i$ is (the normalization of) one of the $(\dim \rho)^2$ eigenvectors which are entries of $[\rho]$. Hence,
\[W(\rho) \neq 0 \Rightarrow \rho \geq (n-k,1^k),\]
and moreover, \(W((n-k,1^k)) \neq 0\).

  Using the coefficients $d_1, \ldots, d_{q_k-1}$
  in Theorem \ref{thm:hoffman-linear}, we have $\lambda_1 = 1$
  and $\lambda_t = \omega$ for $1 < t < q_k$. Since \(T\) is an independent set in the graph \(G=\Gamma_k\), we have
  \begin{eqnarray*}
   0 & = & \sum_{\rho \geq (n-k,1^k)} \lambda_{\rho} W(\rho)\\
& = & \alpha^2 +
    \omega(\alpha - \alpha^2 - W((n-k,1^k)) )+
    \lambda W((n-k,1^k)).
  \end{eqnarray*}
By definition of \(\omega\), we have $\omega(\alpha - \alpha^2) = - \alpha^2$; combining this with the above yields $\lambda= \omega$, as required.

{\bf Step 3 (bounding the coefficients $ d_j$) :}  Recall that $M = N_1 \tilde{C}^t N_2$, and by Theorem
  \ref{thm:minor-is-invertible-2}, \(\tilde{C}\) is invertible and independent of \(n\). Hence, the entries of \((\tilde{C}^t)^{-1}\) are uniformly bounded by some function of \(k\) alone. Since $Md =
  b$, we have 
\begin{equation}
\label{eq:d}
 d = M^{-1} b = (N_2)^{-1} (\tilde{C}^{t})^{-1} (N_1^{-1}b).
\end{equation}
We now proceed to bound uniformly the entries of the vector $N_1^{-1} b$. We have
$$(N_1^{-1})_{ij} = \delta_{ij} \dim[\phi_i],$$
and therefore
$$(N_1^{-1} b)_i =
  \dim[\phi_i] b_i.$$
For \(i=1\), we have
$$(N_1^{-1} b)_1 =
  \dim[\phi_1] b_1 = 1 \cdot 1 = 1.$$
For $i>1$, we have
\[\dim[\phi_i] + 1 \leq \dim(M^{(n-k,1^k)}) = n(n-1)\ldots(n-k+1),\]
since both \(\phi_i\) and the trivial representation are constituents of \(M^{(n-k,1^k)}\). Hence,
\[|(N_1^{-1} b)_i| = |\dim[\phi_i]||\omega| \leq 1.\]

We now bound uniformly the entries of the matrix $(N_2)^{-1}$. Using Lemma \ref{lemma:cycle-measure}, we have
  \[ (N_{2}^{-1})_{ij} = \frac{\delta_{ij}}{|X_j|} \leq \frac{(n-2k-1)(2k-1)^{2k-1}}{n!} \leq \frac{b_{k}}{(n-1)!},\]
where \(b_k >0\) depends only on \(k\).

Combining these bounds with (\ref{eq:d}), we see that there exists \(B_k>0\) depending only on \(k\) such that 
\[|d_j| \leq \frac{B_k}{(n-1)!}\ \forall j < q_k,\]
completing the proof.
\end{proof}

Next, we need two more lemmas to assist in the proof of Theorem
\ref{thm:medreps}.

\begin{lemma} \label{thm:cauchy-schwarz}
  Let \(G\) be a finite group, let \(X \subset G\) be inverse-closed and conjugation-invariant, and let $\Cay(G,X)$ be the Cayley graph on \(G\) with generating set $X$. Let $\rho$
  be an irreducible representation of $G$ with dimension $d$, and let $\lambda_{\rho}$ be
  the corresponding eigenvalue of \(\Cay(G,X)\), as in Theorem
  \ref{thm:cayley-spectrum}. Then
  $$
    |\lambda_{\rho}| \leq \frac{\sqrt{|G||X|}}{d}.
  $$
\end{lemma}
\begin{proof}
  Since the irreducible characters of \(G\) are orthonormal, we have
  \[
  \langle \chi_{\rho},\chi_{\rho} \rangle = \frac{1}{|G|} \sum_{g \in G} |\chi_{\rho}(g)|^2 = 1.
  \]
  By the Cauchy-Schwarz inequality,
  \[|\langle \chi_{\rho}, 1_X \rangle| \leq \sqrt{ \langle \chi_\rho,\chi_{\rho} \rangle\langle 1_X,1_X\rangle } = \sqrt{|X|/|G|}.\]
  Substituting into Equation \ref{cayley-eigenvalue}, we obtain
  \[
  |\lambda_{\rho}| = \frac{|G| |\langle \chi_{\rho},1_X \rangle|}{d} \leq \frac{\sqrt{|G||X|}}{d},
  \]
  as required.
\end{proof}

Combining this with Lemma \ref{lemma:cycle-measure} yields:
\begin{lemma} \label{thm:high-dimension-eigenvalue}
  Let \(X\) be a conjugacy-class of \(S_n\), with cycle-type having a cycle of length at least \(n-2k-1\). Let $\Cay(S_n,X)$ be the Cayley graph on \(S_n\) generated by \(X\). For any \(C_k >0\), there exists \(D_k >0\), depending only on \(k\), such that if $\rho$ is an irreducible representation of dimension at least $C_k
  n^{k+1}$, then the corresponding eigenvalue \(\lambda_\rho\) of $\Cay(S_n,X)$ satisfies
  \[|\lambda_{\rho}| < D_{k}\frac{(n-1)! |\omega_{n,k}|}{\sqrt{n}}.\]
\end{lemma}

\begin{proof}
  Note that
  $|\omega_{n,k}| = \Theta(1/n^k)$. By choosing \(D_k\) large enough, we may assume that \(n>4k+2\), so \(n-2k-1 > n/2\), and therefore the hypotheses of Lemma \ref{lemma:cycle-measure} hold. Assume that $\dim \rho \geq C_k n^{k+1}$. Then

  \begin{eqnarray*}
    |\lambda_{\rho}|
    & \leq &
    \frac{\sqrt{|G||X|}}{C_k n^{k+1}} \\
    & \leq &
    \frac{\sqrt{2n!(n-1)!}}{C_k n^{k+1}} \\
    & = &
    \frac{\sqrt{2}(n-1)!}{C_k n^{k+1/2}} \\
    & \leq &
    D_{k}\frac{(n-1)! |\omega_{n,k}|}{\sqrt{n}},
  \end{eqnarray*}
as required.
\end{proof}

We can now prove Theorem \ref{thm:medreps}.

\begin{proof}[Proof of Theorem \ref{thm:medreps}:]
  Assume the hypotheses of Theorem \ref{thm:medreps}. By Lemma \ref{thm:dimension-for-short-2}, we have \(\dim[\rho] \geq E_{k} n^{k+1}\). Hence, we see that
    \begin{eqnarray*}
|\lambda_{\rho}| & =& \left|\sum_{j=1}^{q_{k}-1} d_j\lambda_{\rho}^{(j)}\right|\\
& \leq & (q_k-1) \max_j |d_j| \max_j |\lambda_{\rho}^{(j)}|\\
& \leq & (q_k-1) \frac{B_k}{(n-1)!} D_k \frac{(n-1)! |\omega_{n,k}|}{\sqrt{n}}\\
& = & c_{k}|\omega|/\sqrt{n},
    \end{eqnarray*}
  as required.
\end{proof}

\subsection{The linear combination of Cayley graphs}
Finally, we can now construct the linear combination of Cayley graphs we will use to
prove our main theorem.

\begin{theorem} \label{thm:yeven-is-ok}
  There exists a linear combination \(\yeven\) of Cayley graphs generated by conjugacy-classes of even permutations within \(\fpf_k\), such that its eigenvalues
  are as described in the first line of Table 1, i.e.,
  \begin{itemize}
 \item $\lambda_{(n)} = 1$,
\item $\lambda_{\rho} = \omega$
  for each fat $\rho \neq (n)$,
\item $\lambda_{(1^n)} = 1$,
\item $\lambda_{\rho} = \omega$ for each tall $\rho \neq (1^n)$, and
\item $|\lambda_{\rho}| \leq c_{k} |\omega| / \sqrt{n} \leq o(|\omega|)$ for each medium $\rho$, where \(c_k>0\) depends only on \(k\).
\end{itemize}
\end{theorem}
\begin{proof}
  Recall that for each partition \(\phi_j > (n-k,1^k)\), exactly one of \(\phi_j\) and \(\Split(\phi_j)\) is even. For each partition \(\phi_j > (n-k,1^k)\), define
\[\mu_j =
  \left\{
  \begin{array}{rl}
    \phi_j,  & \textrm{if }X_{\phi_j} \textrm{ is an even conjugacy class}; \\
    \textrm{Split}(\phi_j), & \textrm{if }X_{\phi_j} \textrm{ is an odd conjugacy class}.
  \end{array}
  \right.\]
Then each \(X_{\mu_j}\) consists of even permutations. Take
\[Y_{\textrm{even}} = \sum_{j=1}^{q_k-1} d_j \Cay(S_n,X_{\mu_j}),\]
where the \(d_j\)'s are as defined in Theorem \ref{thm:linear-exist}. Then we have \(\lambda_{(n)} = 1\), and for each fat \(\rho \neq (n)\), we have \(\lambda_{\rho} = \omega\). By Theorem \ref{thm:transpose-sign}, for any partition \(\rho\), we have \(\chi_{\rho^{t}} = \chi_{\rho} \cdot \sgn\). So by equation (\ref{eq:linearcomb}), for any partition \(\rho\), we have \(\lambda_{\rho^t} = \lambda_{\rho}\), since our Cayley graphs are all generated by conjugacy-classes of even permutations. Therefore, \(\lambda_{(1^n)} = 1\), and \(\lambda_{\rho} = \omega\) for each tall \(\rho \neq (1^n)\). By Theorem \ref{thm:medreps}, \(|\lambda_\rho| \leq c_{k} |\omega| / \sqrt{n} \leq o(|\omega|)\) for each medium \(\rho\), completing the proof.
\end{proof}
\begin{theorem} \label{thm:yodd-is-ok}
  There exists a linear combination \(\yodd\) of Cayley graphs generated by conjugacy-classes of odd permutations within \(\fpf_k\), such that its eigenvalues satisfy
  \begin{itemize}
 \item $\lambda_{(n)} = 1$,
\item $\lambda_{\rho} = \omega$
  for all fat $\rho \neq (n)$,
\item $\lambda_{(1^n)} = -1$,
\item $\lambda_{\rho} = -\omega$ for each tall $\rho \neq (1^n)$, and
\item $|\lambda_{\rho}| \leq c_{k} |\omega| / \sqrt{n} \leq o(|\omega|)$ for each medium $\rho$, where \(c_k>0\) depends only on \(k\).
\end{itemize}
\end{theorem}
\begin{proof}
  For each partition \(\phi_j > (n-k,1^k)\), define
\[\mu_j =
  \left\{
  \begin{array}{rl}
    \phi_j,  & \textrm{if }X_{\phi_j} \textrm{ is an odd conjugacy class}; \\
    \textrm{Split}(\phi_j), & \textrm{if }X_{\phi_j} \textrm{ is an even conjugacy class}.
  \end{array}
  \right.\]
Then each \(X_{\mu_j}\) consists of odd permutations. Take
\[\yodd = \sum_{j=1}^{q_k-1} d_j \Cay(S_n,X_{\mu_j}),\]
where the \(d_j\)'s are as defined in Theorem \ref{thm:linear-exist}. Then we have \(\lambda_{(n)} = 1\), and for each fat \(\rho \neq (n)\), we have \(\lambda_{\rho} = \omega\). This time, our Cayley graphs are all generated by odd permutations, so we have \(\lambda_{\rho^t} = -\lambda_{\rho}\) for every partition \(\rho\). Hence, \(\lambda_{(1^n)} = -1\), and \(\lambda_{\rho} = -\omega\) for each tall \(\rho \neq (1^n)\). Again, by Theorem \ref{thm:medreps}, \(|\lambda_\rho| \leq c_{k} |\omega| / \sqrt{n} \leq o(|\omega|)\) for each medium \(\rho\), completing the proof.
\end{proof}
\begin{theorem} \label{thm:y-is-ok}
  There exists a linear combination \(Y\) of Cayley graphs generated by conjugacy-classes within \(\fpf_k\), such that its eigenvalues satisfy:
\begin{itemize}
 \item $\lambda_{(n)} = 1$,
\item $\lambda_{\rho} = \omega$
  for each fat $\rho \neq (n)$,
\item $\lambda_{\rho} = 0$ for each tall $\rho$, and
\item $|\lambda_{\rho}| \leq c_{k} |\omega| / \sqrt{n} \leq o(|\omega|)$ for each medium $\rho$, where \(c_k>0\) depends only on \(k\).
\end{itemize}
\end{theorem}
\begin{proof}
Set \(Y = \tfrac{1}{2}(Y_{\textrm{even}}+Y_{\textrm{odd}})\).
\end{proof}

\begin{proof}[Proof of Theorems \ref{main-theorem} and \ref{thm:span-is-maximal}:] If \(n\) is sufficiently large depending on \(k\), then \(c_{k}|\omega|/\sqrt{n} < |\omega|\), so \(\omega\) is both the minimum eigenvalue of \(Y\) and the second-largest in absolute value. By our choice of \(\omega\), applying the generalized Hoffman Theorem \ref{thm:hoffman-linear} to \(Y\), we see that if \(I \subset S_n\) is \(k\)-intersecting, then \(|I| \leq (n-k)!\), and that if equality holds, then the characteristic function \(1_{I} \in V_k\). Similarly, applying Theorem \ref{thm:Xhoffman} to \(Y\), we see that if \(I,J \subset S_n\) are \(k\)-cross-intersecting, then $|I||J| \leq ((n-k)!)^{2}$, and that if equality holds, then \(1_I,1_J \in V_k\).
\end{proof}

\section{Boolean functions} \label{sec:boolean-functions}
Recall that $V_k$ is the linear space of real-valued functions on \(S_n\) whose Fourier
transform is supported only on irreducible representations $\geq (n-k,1^k)$. Theorem \ref{thm:cosets-span} states that \(V_k\) is spanned by the
$k$-cosets of \(S_n\). We now wish to characterize the Boolean functions in \(V_k\), proving (a strengthened version of) Theorem \ref{thm:span-is-disjoint-union}.

We will show that every {\em non-negative} function in $V_k$ can be written
as a linear combination of the characteristic functions of $k$-cosets {\em with non-negative coefficients},
and every $0/1$ valued function in \(V_k\) can be written with $0/1$ coefficients.

Let $\A_k$ be the set of ordered $k$-tuples of distinct numbers between
$1$ and $n$, and let \((n)_k = |\mathcal{A}_k| = \frac{n!}{(n-k)!} = n(n-1)\ldots(n-k+1)\). We will prove the following
\begin{theorem} \label{thm:linear-combination-is-union}
Let $f \in V_k$ be non-negative. Then there exist non-negative coefficients
$(b_{\alpha,\beta})_{\alpha,\beta \in \A_k}$ such that $f = \sum b_{\alpha,\beta}1_{T_{\alpha \mapsto \beta}}.$
Furthermore, if $f$ is Boolean, then $f$ is the characteristic function of a disjoint union of $k$-cosets.
\end{theorem}
For didactic reasons, we begin by dealing with the case $k=1$.

\begin{theorem} \label{thm:linear-second-proof-1}
  If $f \in V_1$ is nonnegative, then there exist $b_{i,j} \geq 0$
  such that $f = \sum_{i,j} b_{i,j} 1_{T_{i \mapsto j}}$.
\end{theorem}

\begin{proof}
We say a real \(n \times n\) matrix \(A = (a_{i,j})_{i,j \in [n]}\) \textit{represents} a function \(f \in V_1\) if
\[f = \sum_{i,j}a_{i,j} 1_{T_{i \mapsto j}},\]
or equivalently,
\[f(\sigma) = \sum_{i=1}^{n} a_{i,\sigma(i)}\quad\forall \sigma \in S_{n}.\]
Let \(A\) be a matrix representing \(f\). Our task is to find a {\em non-negative} matrix \(B\) which also represents \(f\). Let \(x_1,\ldots,x_n\) and \(y_1,\ldots,y_n\) be real numbers such that
\[\sum_{i=1}^{n}x_i + \sum_{j=1}^{n} y_j= 0.\]
Let \(B = (b_{i,j})_{i,j \in [n]}\) be the matrix produced from \(A\) by adding \(x_i\) to row \(i\) for each \(i\), and then adding \(y_j\) to column \(j\) for each \(j\), i.e.
\[b_{i,j} = a_{i,j}+x_i+y_j\quad (i,j \in [n]).\]
Note that \(B\) also represents \(f\). We wish to show that there exists a choice of \(x_i\)'s and \(y_j\)'s such that the matrix \(B\) has all its entries non-negative, i.e. we wish to solve the system of inequalities
\begin{equation}\label{eq:1cosetinequalities}
x_{i}+y_{j} \geq -a_{i,j}\ (1 \leq i,j \leq n)\quad \textrm{subject to}\quad \sum_{i}x_{i} + \sum_{j}y_{j}=0.
\end{equation}
By the strong duality theorem of linear programming (see \cite{matousek}), this is unsolvable if and only if there exist \(c_{i,j} \geq 0\) such that
\[\sum_{j}c_{i,j} = 1\ (1 \leq i \leq n),\quad \sum_{i}c_{i,j} = 1 \ (1 \leq j \leq n),\quad \sum_{i,j}c_{i,j}a_{i,j} < 0.\]
Suppose for a contradiction that this holds. The matrix \(C = (c_{i,j})_{i,j \in [n]}\) is bistochastic, and therefore by Birkhoff's theorem (see \cite{birkhoff}), it can be written as a convex combination of permutation matrices,
\[C = \sum_{t=1}^{r} s_{t} P_{\sigma_{t}},\]
where \(s_t \geq 0\ (1 \leq t \leq r)\), \(\sum_{t=1}^{r}s_t = 1\), \(\sigma_1,\ldots,\sigma_t \in S_n\), and \(P_{\sigma}\) denotes the permutation matrix of \(\sigma\), i.e. \((P_{\sigma})_{i,j} = 1_{\{\sigma(i)=j\}}\). But then \(\sum_{i,j}c_{i,j}a_{i,j}\) is a convex combination of values of \(f\):
\begin{equation}
 \label{eq:birkhoffsum}
\sum_{i,j} c_{i,j}a_{i,j} = \sum_{t=1}^{r}s_{t}\sum_{i,j} (P_{\sigma_{t}})_{i,j} a_{i,j} = \sum_{t=1}^{r}s_{t} \sum_{i=1}^{n}a_{i,\sigma_{t}(i)} = \sum_{t=1}^{r} s_{t} f(\sigma_{t})
\end{equation}
and is therefore non-negative, a contradiction. Hence, the system (\ref{eq:1cosetinequalities}) is solvable, proving the theorem.
\end{proof}

The following corollary is immediate.
\begin{cor}
If $f \in V_1$ is Boolean, then it is the characteristic function of a disjoint union of 1-cosets.
\end{cor}
\begin{proof}
By induction on the number of non-zero values of \(f\). Let \(f \in V_1\) be a Boolean function, and suppose the statement is true for all Boolean functions with fewer non-zero values. By Theorem \ref{thm:linear-second-proof-1}, there exist \(b_{i,j} \geq 0\) such that
\[f= \sum_{i,j}b_{i,j} 1_{T_{i \mapsto j}}.\]
Choose \(i,j\) such that \(b_{i,j} > 0\). Then \(f >0\) on \(T_{i \mapsto j}\), so \(f = 1\) on \(T_{i,j}\). Hence, \(f-1_{T_{i \mapsto j}}\) is also Boolean, with fewer non-zero values than \(f\), and so by the induction hypothesis, it is the characteristic function of a disjoint union of \(1\)-cosets. Hence, the same is true of \(f\), as required. 
\end{proof}

We now extend this proof to the case $k>1$, proving Theorem
\ref{thm:linear-combination-is-union}. However, some preliminaries are
necessary.

Let \(A\) be an \((n)_{k} \times (n)_{k}\) matrix with rows and columns indexed by \(\mathcal{A}_{k}\). Choose an ordering of the \(k\) coordinates of the \(k\)-tuples, or equivalently a permutation \(\pi \in S_{k}\), and consider the natural lexicographic ordering on \(\mathcal{A}_{k}\) induced by this ordering, i.e. \(\alpha < \beta\) iff \(\alpha_{\pi(m)} < \beta_{\pi(m)}\), where \(m = \min\{l:\alpha_{\pi(l)} \neq \beta_{\pi(l)}\}\). This lexicographic ordering on \(\mathcal{A}_{k}\) recursively partitions \(A\) into blocks: first it partitions \(A\) into \(n^{2}\) \((n-1)_{k-1} \times (n-1)_{k-1}\) blocks \(B_{i,j}\) according to the \(\pi(1)\)-coordinate of each \(k\)-tuple; then it partitions each block \(B_{i,j}\) into \((n-1)^{2}\) \((n-2)_{k-2} \times (n-2)_{k-2}\) sub-blocks according to the \(\pi(2)\)-coordinate of each \(k\)-tuple, and so on.

The following two recursive definitions concerning such matrices will be crucial.
\begin{definition}
We define a {\em \(k\)-line} in an \((n)_k \times (n)_k\) matrix \(A\) as follows. If \(k=1\), i.e. \(A\) is an \(n \times n\) matrix, a 1-line in \(A\) is just a row or column of \(A\). If \(k > 1\), and \(A\) is an \((n)_k \times (n)_k\) matrix, a \(k\)-line in \(A\) is given by choosing an ordering \(\pi\) of the \(k\)-coordinates, partitioning \(A\) into \(n^2\) blocks according to the \(\pi(1)\)-coordinate of each \(k\)-tuple as above, choosing a row or column of \((n-1)_{k-1} \times (n-1)_{k-1}\) blocks \(B_{i,j}\), and then taking a union of \((k-1)\)-lines, one from each block.
\end{definition}

\begin{definition}
We say that an \((n)_k \times (n)_k\) matrix $A$ is {\em \(k\)-bistochastic} if the following holds. If \(k=1\), an \(n \times n\) matrix \(A\) is {\em 1-bistochastic} if it is bistochastic in the usual sense, i.e. all its entries are non-negative, and all its row and column sums are 1. If $k>1$,
an $(n)_{k} \times (n)_k$ matrix \(A\) is {\em $k$-bistochastic} if,
for any partition into $n^2$ blocks $B_{i,j}$ of size $(n-1)_{k-1}
\times (n-1)_{k-1}$ according to a lexicographic order on \(\mathcal{A}_k\)
induced by any of the $k!$ orderings of the coordinates, there exists
a bistochastic \(n \times n\) matrix $R = (r_{i,j})$ and $n^2$
$(k-1)$-bistochastic matrices $M_{i,j}$ of order $(n-1)$, such that
$B_{i,j} = r_{i,j} M_{i,j}$.
\end{definition}

The following two self-evident observations indicate the relevance of these two notions.
\begin{Observation}\label{obs:linesum}
An $(n)_k \times (n)_k$ matrix with non-negative entries is $k$-bistochastic
if and only if the sum of the entries of every $k$-line in it
(coming from any of the $k!$ recursive partitions into blocks) is 1.
\end{Observation}
\begin{Observation}\label{obs:kcosetpartition}
Given a \(k\)-line \(\ell\), the corresponding \(k\)-cosets \(\{T_{\alpha \mapsto \beta}:\ (\alpha,\beta) \in \ell\}\)
partition \(S_{n}\).
\end{Observation}

Any permutation $\sigma \in S_n$ induces a permutation of \(\mathcal{A}_{k}\); we write \(P_{\sigma}^{(k)}\) for the corresponding $(n)_k \times (n)_k$ permutation matrix, i.e.
\[(P_{\sigma}^{(k)})_{\alpha,\beta} = 1_{\{\sigma(\alpha)=\beta\}}.\]
These $n!$ permutations are only a small
fraction of all $((n)_k)!$ permutations of \(\mathcal{A}_k\). It is not hard to show that any \(P_{\sigma}^{(k)}\) is $k$-bistochastic. It
is a bit harder to show that any other permutation matrix is not
$k$-bistochastic (for this it is important to note that we demand
taking into account lexicographic orderings induced by all possible
orderings of the coordinates.)

The authors thank Siavosh Benabbas for suggesting the correct
formulation of the following theorem and for help in proving it.

\begin{theorem} (Generalized Birkhoff theorem)[Benabbas, Friedgut, Pilpel]\\
\label{thm:genbirkhoff}
An $(n)_k \times (n)_k$ matrix $M$ is $k$-bistochastic if and only if it is
a convex combination of \((n)_{k} \times (n)_{k}\) matrices of permutations of \(\mathcal{A}_{k}\) which are induced by permutations of $[n]$.
\end{theorem}
\begin{proof}
By induction on \(k\).

For $k=1$, this is Birkhoff's theorem.

For the induction step, let $M = r_{i,j} \cdot M_{i,j}$ be a block
decomposition of $M$ as in the definition of $k$-bistochasticity,
according to the natural ordering of the coordinates, i.e. \(\pi=\textrm{id}\). By Birkhoff's
theorem, $R = (r_{i,j})$ is a convex combination of permutation
matrices, and therefore it is either a permutation matrix, or else it can be written as a convex combination $R = sP + (1-s)T$, with $P$ a permutation
matrix, and $T$ a bistochastic matrix with more zero entries than
$R$. Treating $P$ separately, and proceeding in this manner by induction, we may
assume that $R$ is a permutation matrix, and without loss of
generality, we may assume that $R = I$, the identity matrix.  So in $M$, every non-zero entry is indexed by a pair of $k$-tuples of
the form $((i,*,*,\ldots,*),(i,*,*,\ldots,*))$ for some $i$.

Now, we reorder the rows and columns of $M$ with a lexicographic order on \(\mathcal{A}_{k}\)
determined by the transposition \(\pi=(1\ 2)\). Consider now an off-diagonal block in this ordering of the
form $((*,i,*,*,\ldots,*),(*,j,*,*,\ldots,*))$ with $i \neq j$. It
breaks into $(n-1)^2$ sub-blocks of size $(n-2)_{k-2} \times
(n-2)_{k-2}$, but only $n-2$ of them have nonzero entries, namely the sub-blocks
of the form $ ((l,i,*,*,\ldots,*),(l,j,*,*,\ldots,*))$ for some $l \neq i,j$. Hence, the block has a row of
zero sub-blocks, namely the sub-blocks
\[((j,i,*,*,\ldots,*),(m,j,*,*,\ldots,*))\]
for \(m \neq j\), and a column of zero sub-blocks, namely the sub-blocks
\[((m,i,*,*,\ldots,*),(i,j,*,*,\ldots,*))\]
for \(m \neq i\). Since the block is \((k-1)\)-bistochastic, it must be the zero matrix.

As for the diagonal blocks of \(M\) in the new ordering of $\mathcal{A}_{k}$,
blocks of the form $((*,i,*,*,\ldots,*),(*,i,*,*,\ldots,*))$, they can
only have nonzero sub-blocks on the diagonal. Since they are
$(k-1)$-bistochastic, by the induction hypothesis they are each equal
to a convex linear combination of \(P^{(k-1)}_{\sigma}\)'s where each \(\sigma\) is a permutation of $[n]\setminus \{i\}$. But if any of these \(\sigma\)'s has \(\sigma(l) \neq l\) for some \(l\), then the off-diagonal sub-block
\[((l,i,*,*,\ldots,*),(\sigma(l),i,*,*,\ldots,*))\]
is non-zero, a contradiction. Hence, each permutation \(\sigma\) must be the identity, and therefore all the diagonal blocks are copies of the $(n-1)_{k-1} \times (n-1)_{k-1}$ identity matrix. Hence, $M$ is the
identity matrix, and we are done.
\end{proof}
We now can prove Theorem \ref{thm:linear-combination-is-union}.

\begin{proof}[Proof of Theorem \ref{thm:linear-combination-is-union}:]
  We follow the lines of the proof of Theorem
  \ref{thm:linear-second-proof-1}. Let $f \in V_k$ be non-negative, and let \(M = (M_{\alpha,\beta})_{\alpha,\beta \in \mathcal{A}_{k}}\) be an \((n)_k \times (n)_k\) matrix (with rows and columns indexed by \(\mathcal{A}_k\)) which represents \(f\), meaning that
\[f = \sum_{\alpha,\beta \in \mathcal{A}_{k}}M_{\alpha,\beta}1_{T_{\alpha \mapsto \beta}},\]
i.e.
\[f(\sigma) = \sum_{\alpha \in \mathcal{A}_{k}} M_{\alpha,\sigma(\alpha)}\quad \forall \sigma \in S_{n}.\]
Let $\Ell_k$ be the set
  of all $k$-lines. Recall from Observation \ref{obs:kcosetpartition}
  that every $k$-line corresponds to a partition of $S_n$ into
  $k$-cosets. Therefore, adding a constant $x \in \mathbb{R}$ to a single $k$-line in $M$
  corresponds to increasing $f$ by $x$. We therefore associate with
  every line $\ell \in \Ell_k$ a variable $x_\ell$. We wish to solve the following system of linear inequalities
\[\sum_{\ell :(\alpha,\beta) \in \ell}x_{\ell} \geq -M_{\alpha,\beta}\ (\ell \in \mathcal{L}_{k}) \quad \textrm{subject to}\quad\sum_{\ell \in \Ell_{k}}x_{\ell}=0.\]

Again, by the strong duality theorem of linear programming, this is possible unless there exists an \((n)_{k} \times (n)_{k}\) matrix \(C = (c_{\alpha,\beta})_{\alpha,\beta \in \mathcal{A}_{k}}\) with non-negative entries, such that
\[\sum_{(\alpha,\beta)\in \ell} c_{\alpha,\beta} = 1\quad \forall \ell \in \Ell_{k},\]
and
\[\sum_{\alpha,\beta \in \mathcal{A}_{k}} c_{\alpha,\beta} M_{\alpha,\beta} <0.\]
 By Observation \ref{obs:linesum}, \(C\) is \(k\)-bistochastic. By the Generalized Birkhoff Theorem,
  $C$ is in the convex hull of the permutation matrices induced by
  permutations of $[n]$. This means, as before, that
  \[\sum_{\alpha,\beta \in \mathcal{A}_{k}} c_{\alpha,\beta}M_{\alpha,\beta}\]
   is a convex linear
  combination of values of $f$, and is therefore non-negative, a contradiction. Therefore, the system of linear inequalities above is solvable, proving the first part of the theorem. The second part follows as in the \(k=1\) case, by induction on the number of non-zero values of \(f\).
\end{proof}

This completes the proof of Theorems \ref{thm:final-result} and \ref{thm:X-main}.

\section{Open Problems}
In their landmark paper \cite{ak}, Ahlswede and Khachatrian characterized the largest \(k\)-intersecting subsets of \([n]^{(r)}\) for every value of \(k\), \(r\) and \(n\). In our opinion, the most natural open problem in the area is to characterize the largest \(k\)-intersecting subsets of \(S_{n}\) for every value of \(n\) and \(k\). We make the following

\begin{conj}
\label{conj:full}
A \(k\)-intersecting family in \(S_{n}\) with maximum size must be a `{\em double translate}' of one of the families
\[M_{i} = \{\sigma \in S_{n}:\ \sigma \textrm{ has }\geq k+i \textrm{ fixed points in } [k+2i]\}\ (0 \leq i \leq (n-k)/2),\]
i.e. of the form \(\sigma M_i \tau\) where \(\sigma,\tau \in S_n\).
\end{conj}

This would imply that the maximum size is \((n-k)!\) for \(n > 2k\). It is natural to ask how large \(n\) must be for our method to give this bound, i.e. when there exists a weighted graph \(Y\) which is a real linear combination of Cayley graphs on \(S_{n}\) generated by conjugacy-classes in \(\textrm{FPF}_k\), such that \(Y\) has maximum eigenvalue \(1\) and minimum eigenvalue
\[\omega_{n,k} = -\frac{1}{n(n-1)\ldots(n-k+1)-1}.\]
It turns out that this need not be the case when \(n = 2k+2\); indeed, it fails for \(k=2\) and \(n=6\). We believe that new techniques will be required to prove Conjecture \ref{conj:full}.\\

\subsection*{Acknowledgements}
The authors would like to thank Noga
Alon, Siavosh Benabbas, Imre Leader, Timothy Gowers, Yuval Roichman, and
John Stembridge for various useful remarks.

\subsection*{Erratum, added 7th July 2017}
Recently, Yuval Filmus \cite{filmus} has pointed out that Theorem \ref{thm:linear-combination-is-union} is false for each $k \geq 2$; the above proof of Theorem \ref{thm:genbirkhoff} contains a hole, and in fact Theorem \ref{thm:genbirkhoff} is false for each $k \geq 2$. (Both theorems are true in the special case $k=1$.) Counterexamples are given in \cite{filmus}. The equality part of Theorem \ref{thm:final-result} can of course be deduced from Theorem \ref{thm:ellis-stability}, instead of from the (false) Theorem \ref{thm:linear-combination-is-union}. Moreover, the equality part of Theorem \ref{thm:X-main} can be deduced from a stability result for cross-intersecting families of permutations, Theorem 10 in \cite{kstability}. The proof of the upper bound in Theorem \ref{thm:final-result} is of course unaffected.

\bibliography{tintersecting}

\begin{thebibliography}{10}

\bibitem{ak}
Rudolf Ahlswede and Levon~H. Khachatrian.
\newblock The complete intersection theorem for systems of finite sets.
\newblock {\em European Journal of Combinatorics}, 18(2):125--136, 1997.

\bibitem{AKKMS}
Noga Alon, Haim Kaplan, Mchael Krivelevich, Dalia Malkhi, and Julien Stern.
\newblock Scalable secure storage when half the system is faulty.
\newblock In {\em 27th International Colloquium on Automata, Languages and
  Programming}, pages 576--587, 2000.

\bibitem{babai}
L\'{a}szl\'{o} Babai.
\newblock Spectra of cayley graphs.
\newblock {\em Journal of Combinatorial Theory, Series B}, 1979.

\bibitem{birkhoff}
G.~Birkhoff.
\newblock Three observations on linear algebra.
\newblock {\em Univ. Nac. Tucum\'an. Revista A.}, 5:147--151, 1946.

\bibitem{Bourgain}
Jean Bourgain.
\newblock On the distribution of the {Fourier} spectrum of {B}oolean functions.
\newblock {\em Israel Journal of Mathematics}, 131:269--276, 2002.

\bibitem{cameron-ku}
Peter~J. Cameron and C.~Y. Ku.
\newblock Intersecting families of permutations.
\newblock {\em European Journal of Combinatorics}, 24:881--890, 2003.

\bibitem{diaconis-1980}
Persi Diaconis and Mehrdad Shahshahani.
\newblock Generating a random permutation with random transpositions.
\newblock {\em Probability Theory and Related Fields}, 57(2):159--179, 1981.

\bibitem{kstability}
David Ellis.
\newblock Stability for {\em t}-intersecting families of permutations.
\newblock To appear in {\em Journal of Combinatorial Theory, Series A}.

\bibitem{filmus}
Yuval Filmus.
\newblock A comment on {I}ntersecting {F}amilies of {P}ermutations, preprint,
  ar{X}iv:1706.10146.

\bibitem{frt}
J.~S. Frame, G.~B. Robinson, and R.~M. Thrall.
\newblock The {H}ook graphs of {$S_{n}$}.
\newblock {\em Canadian Journal of Mathematics}, 6:316--325, 1954.

\bibitem{deza-frankl}
Peter Frankl and Mikhail Deza.
\newblock On the maximum number of permutations with given maximal or minimal
  distance.
\newblock {\em Journal of Combinatorial Theory, Series A}, 22(3):352--360,
  1977.

\bibitem{F98}
Ehud Friedgut.
\newblock Boolean functions with low average sensitivity depend on few
  coordinates.
\newblock {\em Combinatorica}, 18(1):27--36, 1998.

\bibitem{FKN}
Ehud Friedgut, Gil Kalai, and Assaf Naor.
\newblock Boolean functions whose {Fourier} transform is concentrated on the
  first two levels and neutral social choice.
\newblock {\em Advances in Applied Mathematics}, 29:427--437, 2002.

\bibitem{friedguttalk}
Ehud Friedgut and Haran Pilpel.
\newblock Intersecting families of permutations: An algebraic approach.
\newblock Talk given at the Oberwolfach Conference in Combinatorics, January
  2008.

\bibitem{godsil-2008}
Chris Godsil and Karen Meagher.
\newblock A new proof of the {E}rd{\H{o}}s-{K}o-{R}ado theorem for intersecting
  families of permutations.
\newblock {\em European Journal of Combinatorics}, 30(2):404--414, 2009.

\bibitem{amgm}
G.H. Hardy, J.E. Littlewood, and G.~P\'olya.
\newblock {\em An Introduction to Inequalities}.
\newblock Cambridge University Press, 1988.

\bibitem{hiltonmilner}
A.J.W. Hilton and E.C. Milner.
\newblock Some intersection theorems for systems of finite sets.
\newblock {\em Quarterly Journal of Mathematics, Oxford}, 18:369--384, 1967.

\bibitem{hoffman}
A.~J. Hoffman.
\newblock On eigenvalues and colorings of graphs.
\newblock {\em Graph Theory and its Applications}, 1969.

\bibitem{isaacs}
I.~Martin Isaacs.
\newblock {\em Character Theory of Finite Groups}.
\newblock Academic Press, 1976.

\bibitem{jordan}
C.~Jordan.
\newblock {\em Trait\'e des Substitutions et des Equations Alg\'ebriques}.
\newblock Gauthier-Villars, Paris, 1870.

\bibitem{kkl}
Jeff Kahn, Gil Kalai, and Nathan Linial.
\newblock The influence of variables on {B}oolean functions (extended
  abstract).
\newblock In {\em FOCS}, pages 68--80. IEEE, 1988.

\bibitem{larose-malvenuto}
Benoit Larose and Claudia Malvenuto.
\newblock Stable sets of maximal size in {K}neser-type graphs.
\newblock {\em European Journal of Combinatorics}, 25(5):657--673, 2004.

\bibitem{leader-private-2008}
Imre Leader.
\newblock Open Problems Session, British Combinatorial Conference, 2005.

\bibitem{matousek}
Ji\v{r}\'i Matou\v{s}ek and Bernd G\"{a}rtner.
\newblock {\em Understanding and Using Linear Programming}.
\newblock Springer-Verlag, Berlin, 2007.

\bibitem{mish96}
S.~P. Mishchenko.
\newblock Lower bounds on the dimensions of irreducible representations of
  symmetric groups and of the exponents of the exponential of varieties of
  {L}ie algebras.
\newblock {\em Matematicheski Sbornik}, 187:83--94, 1996.

\bibitem{renteln}
Paul Renteln.
\newblock On the spectrum of the derangement graph.
\newblock {\em Electronic Journal of Combinatorics}, 14, 2007.
\newblock \#R82.

\bibitem{roichman-1994}
Yuval Roichman.
\newblock Upper bound on the characters of the symmetric groups.
\newblock {\em Inventiones Mathematicae}, 125:451--485, 1996.

\bibitem{sagan}
Bruce~E. Sagan.
\newblock {\em {The Symmetric Group: Representations, Combinatorial Algorithms,
  and Symmetric Functions}}.
\newblock Springer-Verlag, New-York, 1991.
\newblock [2nd revised printing, 2001].

\bibitem{terras}
Audrey Terras.
\newblock {\em Fourier Analysis on Finite Groups and Applications}.
\newblock London Mathematical Society Student Texts, 43. Cambridge University
  Press, 1999.

\end{thebibliography}
\bibliographystyle{plain}

\end{document}